\documentclass{svjour3}                     
\smartqed  
\usepackage{setspace}
\usepackage{algorithm}
\usepackage{algorithmic}

\usepackage{hyperref}
\usepackage{enumerate}
\usepackage{amssymb}
\usepackage{booktabs}
\usepackage{graphicx,fancyhdr}
\usepackage{graphics,color}
\usepackage{subfigure}
\usepackage{booktabs}
\usepackage{graphicx}
\usepackage{amsmath}
\numberwithin{equation}{section}
\numberwithin{theorem}{section}
\numberwithin{lemma}{section}
\numberwithin{remark}{section}

\usepackage{mathptmx}
\begin{document}

\title{Numerical algorithms of the two-dimensional Feynman-Kac equation for reaction and diffusion processes}


\author{Daxin Nie$^{1}$, Jing Sun$^{1}$, Weihua Deng$^{*,1}$
}


\institute{
$^{*}$Corresponding author. E-mail: dengwh@lzu.edu.cn            \\
$^{1}$School of Mathematics and Statistics, Gansu Key Laboratory of Applied Mathematics and Complex Systems, Lanzhou University, Lanzhou 730000, P.R. China \\
}

\date{Received: date / Accepted: date}

\maketitle

\begin{abstract}
This paper provides a finite difference discretization for the backward Feynman-Kac equation, governing the distribution of functionals of the path for a particle undergoing both reaction and diffusion [Hou and Deng, J. Phys. A: Math. Theor., {\bf51}, 155001 (2018)]. Numerically solving the equation with the time tempered fractional substantial derivative and tempered fractional Laplacian consists in discretizing these two non-local operators. Here, using convolution quadrature, we provide a first-order and second-order schemes for discretizing the time tempered fractional substantial derivative, which doesn't require the assumption of the regularity of the solution in time; we use the finite difference method to approximate the two-dimensional tempered fractional Laplacian, and the accuracy of the scheme depends on the regularity of the solution on $\bar{\Omega}$ rather than the whole space. Lastly, we verify the predicted convergence orders and the effectiveness of the presented schemes by numerical examples.

\end{abstract}


\section{Introduction}

The random motion of a particle is a most fundamental and widely appeared natural phenomena. The so-called particle can be a really physical one or an abstract one, e.g., stock market. Stochastic processes  $x_0(t)$ are mathematical models to describe this phenomena. In a long history, the Wiener process is the most studied and representative stochastic process, the second moment of which is a linear function of time $t$. In the past twenty years, it is found that anomalous dynamics  are ubiquitous in the natural world. The anomalous stochastic processes are  distinguished from Wiener process by the evolution of their second moments with respect to  time $t$, that is, the second moment of the stochastic anomalous process is a non-linear function of $t$ \cite{Metzler2000}.

Currently, stochastic anomalous processes are hot research topics, including modeling, theoretical analysis, and numerical methods. Obtaining the probability density function (PDF)  of the statistical observables plays an important role in studying the stochastic processes; not only it is a useful technique to extract practical messages but also a key strategy to help understand the mechanism of stochastic processes. The functional is one of the most useful and representative statistical observables, defined as
\begin{equation}\label{func1}
A=\int_{0}^{t}U[x_0(\tau)]d\tau,
\end{equation}
where $x_0(t)$ is a trajectory of particle, $U(x_0)$ is a prescribed function depending on specific applications, e.g., one can take $U(x_0)=1$ in a given region and set it to be zero in the rest of the region for \eqref{func1} to study the kinetics of chemical reactions in some given domain \cite{Agmon1984,Carmi2010}. Some important progresses have been made for deriving the governing equation of the PDF of the functional $A$ in Fourier space. The earliest work \cite{Kac1949} for this issue is for the distribution of the functional of Wiener process, completed by Kac in 1949 who was influenced by Feynman's thesis of the derivation of Schr\"odinger's equation. After that the related equations are usually named with the word: Feynman-Kac.   For the stochastic process $x_0(t)$ described by a continuous time random walk (CTRW) with power law waiting time and jump length distributions, Barkai and his collaborators derive the governing equation of the distribution of the corresponding functional in Fourier space \cite{Carmi2010,Turgeman2009}. Because of the finite lifespan and the bounded physical space, sometimes the power law waiting time and jump length distributions of the CTRW need to be tempered \cite{Meerschaert2012}; then the tempered fractional Feynman-Kac equations are derived in \cite{Wu2016}, and the numerical methods for the equations are discussed in \cite{Chen2018,Chen2018-2,Deng2017,Sun2018,Zhang2017,ZhangDeng2017}.  For $x_0(t)$ characterized by Langevin pictures, the readers can refer to \cite{Cairoli2015,Cairoli2017} for the derivation of Feynman-Kac equations.

More recently, the reactions are introduced to the stochastic process $x_0(t)$, which means that the particles perform both diffusion and chemical reaction.  With the exponentially tempered power law distributions of waiting time and jump length for the diffusion, the Feynman-Kac equation for the reaction and diffusion process is derived in \cite{Hou2018}, which includes  the tempered fractional substantial derivative and the tempered fractional Laplacian.
In this paper, we consider the numerical scheme of the two-dimensional backward tempered fractional Feynman-Kac equation for reaction and diffusion processes with homogeneous Dirichlet boundary conditions \cite{Hou2018}, i.e.,
\begin{equation}\label{equFk}
    \begin{aligned}
    \left\{
    \begin{aligned}
        &\frac{\partial}{\partial t}G(\rho,t,\mathbf{x_0})=K ~_0D_{t}^{1-\alpha, \lambda, \mathbf{x_0}} (\Delta+\gamma)^{\frac{\beta}{2}} G(\rho,t,\mathbf{x_0})+(r(\mathbf{x_0})+J\rho U(\mathbf{x_0}))G(\rho,t,\mathbf{x_0})\\
        &~~~~~~~~~~~~~~~~~~~~~~~~~+(\lambda^{\alpha}~_0D_{t}^{1-\alpha, \lambda, \mathbf{x_0}}-\lambda)\left(G(\rho,t,\mathbf{x_0})-e^{J\rho U(\mathbf{x_0})t}e^{r(\mathbf{x_0})t}\right),\\
        &G(\rho,0,\mathbf{x_0})=G_{0}(\rho,\mathbf{x_0}),~~~~~\mathbf{x_0}\in \Omega,\\
        &G(\rho,t,\mathbf{x})=0,~~~~~~\mathbf{x}\in \mathbb{R}^2\backslash\Omega,~~~~0\leq t\leq T.\\
    \end{aligned}
    \right.
    \end{aligned}
\end{equation}
Here, $G(A,t,\mathbf{x_0})$ is the PDF of the functional $A$ at time $t$ with the initial position $\mathbf{x_0}$, and $\rho$ is the Fourier pair of $A$;  $K$ is a positive constant, for convenience, we take $K=1$ in this paper; $0<\alpha<1$, $J=\sqrt{-1}$ and $\lambda>0$; $r(\mathbf{x_0})$ stands for the reaction rate and satisfies $\sup_{\mathbf{x_0}\in \bar{\Omega}}r(\mathbf{x_0})<0$; $~_0D_{t}^{1-\alpha, \lambda, \mathbf{x_0}}$ is the tempered fractional substantial derivative, which is defined by
\begin{equation}\label{equDtime1}
  \begin{aligned}
        &~_0D_{t}^{1-\alpha, \lambda, \mathbf{x_0}}G(\rho,t,\mathbf{x_0})\\
        &=\frac{1}{\Gamma(\alpha)}\left(\frac{\partial}{\partial t}+\lambda-r(\mathbf{x_0})-J\rho U(\mathbf{x_0})\right)\int_{0}^{t}\frac{e^{-(t-\tau)(\lambda-r(\mathbf{x_0})-J\rho U(\mathbf{x_0}))}}{(t-\tau)^{1-\alpha}}G(\rho,\tau,\mathbf{x_0})d\tau.
    \end{aligned}
\end{equation}
And $(\Delta+\gamma)^{\frac{\beta}{2}}$ denotes the two-dimensional tempered fractional Laplacian \cite{Deng:17}, whose definition is
\begin{equation}\label{spatemdef}
(\Delta+\gamma)^{\frac{\beta}{2}}G(\mathbf{x})=-c_{2,\beta}{\rm P.V.}\int_{\mathbb{R}^2}\frac{G(\mathbf{x})-G(\mathbf{y})}{e^{\gamma|\mathbf{x}-\mathbf{y}|}|\mathbf{x}-\mathbf{y}|^{2+\beta}}d\mathbf{y} ~~~~~{\rm for}~~\beta \in(0,2),
\end{equation}
with
\begin{equation*}
c_{2,\beta}=
\left\{
\begin{aligned}
&\frac{1}{2\pi|\Gamma(-\beta)|}         ~~~~~~~~~~~~{\rm for} ~\lambda>0~{\rm and}~\beta\neq 1,\\
&\frac{\beta\Gamma(\frac{2+\beta}{2})}{2^{1-\beta}\pi\Gamma(1-\beta/2)} ~~{\rm for}~\lambda=0~{\rm or}~\beta=1.
\end{aligned}
\right.
\end{equation*}
 P.V. denotes the principal value integral, $\mathbf{x},~\mathbf{y}\in \mathbb{R}^2$, and $\Gamma(t)=\int_{0}^{\infty}s^{t-1}e^{-s}ds$ is the Gamma function.

When numerically solving Eq. (\ref{equFk}), two main problems need to be carefully dealt with. The first one is to discretize  the tempered fractional substantial derivative (\ref{equDtime1}), which is a time-space coupled operator and whose form depends on the initial position $\mathbf{x_0}$; existing methods of discretizing it mainly require that the solution should be a $C^2$ or $C^3$-function in time, such as G-L scheme and $L_1$ scheme \cite{Chen2018,Chen2018-2,Deng2015,Deng2017}. Here, we develop a first-order and second-order schemes based on convolution quadrature introduced by Lubich \cite{Lubich1988,Lubich1988-2}. Theorem \ref{thmBEsemER} and Theorem \ref{thmSBDsemER} in Sec. 4 show that the convergence order only depends on the regularity of source term $f$ instead of the exact solution $G$.
The second problem is to discretize (\ref{spatemdef}); so far, the discretizations of (\ref{spatemdef}) are mainly for the one-dimensional case \cite{Zhang2017,ZhangDeng2017}; our previous work \cite{Sun2018} provides a finite difference scheme for the two-dimensional tempered fractional Laplacian $(\Delta+\gamma)^{\frac{\beta}{2}}$. Here, we modify the discretization according to \cite{Zhang2017}, so that the regularity requirement can be relaxed from the whole space to $\bar{\Omega}$, and the optimal convergence rate is achieved.
Furthermore, we provide a method to construct a preconditioner when we solve relative linear system by Preconditioned Conjugate Gradient (PCG) method.

The framework of convolution quadrature \cite{Jin2016} can be briefly reviewed as follow. Firstly, one can define $\tilde{\mathcal{B}}(\partial_t)v(t)=(\mathcal{B}\ast v)(t)$, where $\partial_t$ denotes time differentiation; and let $\tau$ denote the time step size. The convolution quadrature refers to an approximation of any function of the form $\mathcal{B}\ast v$ as
\begin{equation*}
  (\mathcal{B}\ast v)(t)=\int_0^t \mathcal{B}(t-s)v(s)ds\approx \tilde{\mathcal{B}}(\bar{\partial}_\tau)v(t),
\end{equation*}
where
\begin{equation*}
 \tilde{\mathcal{B}}(\bar{\partial}_\tau)v(t)=\sum_{0\leq j\tau\leq t}d_jv(t-j\tau),~~~~t>0,
\end{equation*}
and the quadrature weights $\{d_j\}_{j=0}^\infty$ are computed from $\tilde{\mathcal{B}}(z)$ denoting Laplace transform  of $\mathcal{B}(t)$, i.e., $\sum\limits_{j=0}^{\infty}d_j\zeta^j=\tilde{\mathcal{B}}(\delta(\zeta)/\tau)$. Here $\delta(\zeta)$ is the quotient of the generating polynomials of a stable and consistent linear multi-step method \cite{Hairer}. In this paper, we take $\delta(\zeta)=1-\zeta$ and $\delta(\zeta)=(1-\zeta)+(1-\zeta)^2/2$ to, respectively, get  first-order and second-order scheme for the tempered fractional substantial derivative (\ref{equDtime1}).

The remainder of the paper is organized as follows. In Sec. 2, we give some preliminaries needed in the paper and derive an equivalent form of  Eq. \eqref{equFk}. In Sec. 3, we combine the weighted trapezoidal rule and the bilinear interpolation to discretize the tempered fractional Laplacian and perform error analysis. In Sec. 4, we use convolution quadrature to discretize the tempered fractional substantial derivative, and then get the first-order and second-order schemes. In Sec. 5, we present the efficient computation of the linear system generated by the  discretization.

\section{Preliminaries and equivalent form of Eq. \eqref{equFk}}
This section introduces some preliminary knowledges and derives the equivalent formulation of Eq. \eqref{equFk}.

\subsection{Preliminaries}
This subsection provides some definitions and properties needed in the paper. Firstly, define the discrete inner product and the discrete norm as
\begin{equation*}
\begin{split}
  &(\mathbf{v},\mathbf{w})=h^2\sum_{i=1}^{M}v_i {w}_i^*,\\
  &\|\mathbf{v}\|_{l^2}=\sqrt{(\mathbf{v},\mathbf{v})},\\
  &\|\mathbf{v}\|_{l^\infty}=\max_{1\leq i\leq M}|v_i|,
\end{split}
\end{equation*}
where $\mathbf{v},\mathbf{w}\in \mathbb{R}^M$; denote
\begin{equation*}
\begin{split}
 &\|v\|_{L^\infty(\Omega)}=\sup_{x\in \Omega}|v(x)|, \\
 &\|v\|^2_{L^2(\Omega)}=\int_{\Omega}|v|^2dx,
\end{split}
\end{equation*}
as the continuous norms, where $|v(x)|^2=v(x)v^*(x)$ and $v^*(x)$ means the conjugate of $v(x)$. Furthermore, we recall some definitions of the tempered fractional integrals and derivatives.
\begin{definition}[Riemann-Liouville tempered fractional integral \cite{Buschman1972,Cartea2007,Li2015}] \label{def1}
 Suppose that the real function $v(t)$ is piecewise continuous on $(a,b)$ and $\alpha> 0, \lambda\geq0$, $v(t)\in L[a,b]$. The Riemann-Liouville tempered fractional integral of order $\alpha$ is defined to be
\begin{equation*}
_{a}I^{\alpha, \lambda}_{t}v(t)=e^{-\lambda t}~_{a}I^{\alpha}_{t}\left(e^{\lambda t}v(t)\right)=\frac{1}{\Gamma(\alpha)}\int^t_{a}(t-\tau)^{\alpha-1}e^{-\lambda(t-\tau)}v(\tau)d\tau,
\end{equation*}
where $_{a}I^{\alpha}_{t}v(t)$ denotes the Riemann-Liouville fractional integral
\begin{equation*}
_{a}I^{\alpha}_{t}v(t)=\frac{1}{\Gamma(\alpha)}\int^t_{a}(t-\tau)^{\alpha-1}v(\tau)d\tau.
\end{equation*}
\end{definition}

\begin{definition}[Riemann-Liouville tempered fractional derivative \cite{Baeumer2010,Cartea2007,Li2015}]
For $n-1<\alpha<n,n\in \mathbb{N}^{+},\lambda\geq 0$, the Riemann-Liouville tempered fractional derivative is defined by
\begin{equation*}
_{a}D^{\alpha, \lambda}_{t}v(t)=e^{-\lambda t}~_{a}D^{\alpha}_{t}\left(e^{\lambda t}v(t)\right)=\frac{e^{-\lambda t}}{\Gamma(n-\alpha)}\frac{d^n}{dt^n}\int^t_{a}\frac{e^{\lambda\tau}v(\tau)}{(t-\tau)^{\alpha-n+1}}d\tau,
\end{equation*}
where $_{a}D^{\alpha}_{t}v(t)$ denotes the Riemann-Liouville fractional derivative and it is described as
\begin{equation*}
_{a}D^{\alpha}_{t}v(t)=\frac{1}{\Gamma(n-\alpha)}\frac{d^n}{dt^n}\int^t_{a}\frac{v(\tau)}{(t-\tau)^{\alpha-n+1}}d\tau.
\end{equation*}
\end{definition}

\begin{remark}
  When $0<\alpha<1$, there hold that
\begin{equation*}
    ~_0D^\alpha_t~_0I^\alpha_t v(t)=v(t)
\end{equation*}
and
\begin{equation*}
    ~_0I^\alpha_t ~_0D^\alpha_t v(t)=v(t).
\end{equation*}
\end{remark}

\begin{definition}[Caputo tempered fractional derivative \cite{Li2015,Samko,Tatar2004}]\label{def2}
For $n-1<\alpha<n,\,n\in \mathbb{N}^{+},\,\lambda\geq 0$, the Caputo tempered fractional derivative is defined as
\begin{equation*}
~^C_{a}D^{\alpha,\lambda}_tv(t)=e^{-\lambda t}~^C_{a}D^{\alpha}_{t}\left(e^{\lambda t}v(t)\right)=\frac{e^{-\lambda t}}{\Gamma(n-\alpha)}\int^t_{a}(t-\tau)^{n-\alpha-1}\frac{d^n(e^{\lambda\tau}v(\tau))}{d \tau^n}d\tau,
\end{equation*}
 where $~^C_{a}D^{\alpha}_{t}v(t)$ denotes the Caputo fractional derivative and it is defined by
 \begin{equation*}
~^C_{a}D^{\alpha}_tv(t)=\frac{1}{\Gamma(n-\alpha)}\int^t_{a}(t-\tau)^{n-\alpha-1}\frac{d^n v(\tau)}{d \tau^n}d\tau.
\end{equation*}
\end{definition}

\begin{proposition}[\cite{Li2015}]\label{procapLap}
  The Laplace transform of the Riemann-Liouville tempered fractional derivative is given by
  \begin{equation*}
    \widetilde{~_0D^{\alpha,\lambda}_{t}v}(z)=(z+\lambda)^\alpha \tilde{v}(z)-\sum_{k=0}^{n-1}(z+\lambda)^k\left(~_0D^{\alpha-k-1}_{t}(e^{\lambda t}v(t))\Big|_{t=0}\right),
  \end{equation*}
 while the Laplace transform of the Caputo tempered fractional derivative is
  \begin{equation*}
    \widetilde{~_0^CD^{\alpha,\lambda}_{t}v}(z)=(z+\lambda)^\alpha \tilde{v}(z)-\sum_{k=0}^{n-1}(z+\lambda)^{\alpha-k-1}\left(\frac{d^k}{dt^k}(e^{\lambda t}v(t))\Big|_{t=0}\right).
  \end{equation*}
\end{proposition}
\subsection{Equivalent form of Eq. \eqref{equFk}}

According to \eqref{equDtime1} and Definition \ref{def1}, we get
\begin{equation}\label{equDtime2}
    \begin{aligned}
        &~_0D_{t}^{1-\alpha, \lambda, \mathbf{x_0}}G(\rho,t,{\mathbf{x_0}})\\
        &=\left(\frac{\partial}{\partial t}+\lambda-r(\mathbf{x_0})-J\rho U(\mathbf{x_0})\right)e^{-(\lambda-r(\mathbf{x_0})-J\rho U(\mathbf{x_0}))t}~_0I^\alpha_t\left(e^{(\lambda-r(\mathbf{x_0})-J\rho U(\mathbf{x_0}))t}G(\rho,t,{\mathbf{x_0}})\right).
    \end{aligned}
\end{equation}
By calculating directly, we obtain
\begin{equation}\label{equDtime3}
  ~_0D_{t}^{1-\alpha, \lambda, \mathbf{x_0}}G(\rho,t,{\mathbf{x_0}})=e^{-(\lambda-r(\mathbf{x_0})-J\rho U(\mathbf{x_0}))t}~_0D_{t}^{1-\alpha}\left(e^{(\lambda-r(\mathbf{x_0})-J\rho U(\mathbf{x_0}))t} G(\rho,t,{\mathbf{x_0}})\right).
\end{equation}
Combining \eqref{equDtime3} with \eqref{equFk}, we have
\begin{equation}\label{equFk1}
    \begin{aligned}
        \frac{\partial}{\partial t}G(\rho,t,{\mathbf{x_0}})=&e^{-(\lambda-r(\mathbf{x_0})-J\rho U(\mathbf{x_0}))t}~_0D_{t}^{1-\alpha}\left(e^{(\lambda-r(\mathbf{x_0})-J\rho U(\mathbf{x_0}))t}(\Delta+\gamma)^{\frac{\beta}{2}} G(\rho,t,{\mathbf{x_0}})\right)\\
        &+\left(r(\mathbf{x_0})+J\rho U(\mathbf{x_0})\right)G(\rho,t,{\mathbf{x_0}})-\lambda G(\rho,t,{\mathbf{x_0}})+\lambda e^{J\rho U(\mathbf{x_0})t}e^{r(\mathbf{x_0})t}\\
        &+\lambda^{\alpha}e^{-(\lambda-r(\mathbf{x_0})-J\rho U(\mathbf{x_0}))t}~_0D_t^{1-\alpha}\left(e^{(\lambda-r(\mathbf{x_0})-J\rho U(\mathbf{x_0}))t}G(\rho,t,{\mathbf{x_0}})\right)\\
        &-\lambda^{\alpha}e^{-(\lambda-r(\mathbf{x_0})-J\rho U(\mathbf{x_0}))t}~_0D_t^{1-\alpha}e^{\lambda t}.
\end{aligned}
\end{equation}
Further simplification leads to
\begin{equation}\label{equFk2}
    \begin{aligned}
        &\left(\frac{\partial}{\partial t}G(\rho,t,{\mathbf{x_0}})+\lambda G(\rho,t,{\mathbf{x_0}})-(r(\mathbf{x_0})+J\rho U(\mathbf{x_0}))G(\rho,t,{\mathbf{x_0}})\right)e^{(\lambda-r(\mathbf{x_0})-J\rho U(\mathbf{x_0}))t}\\
&=~_0D_{t}^{1-\alpha}\left(e^{(\lambda-r(\mathbf{x_0})-J\rho U(\mathbf{x_0}))t} (\Delta+\gamma)^{\frac{\beta}{2}} G(\rho,t,{\mathbf{x_0}})\right)+\lambda e^{\lambda t}\\
&~~~~+\lambda^{\alpha}~_0D_{t}^{1-\alpha}\left(e^{(\lambda-r(\mathbf{x_0})-J\rho U(\mathbf{x_0}))t}G(\rho,t,{\mathbf{x_0}})\right)-\lambda^{\alpha}~_0D_t^{1-\alpha}e^{\lambda t}.
    \end{aligned}
\end{equation}
It is easy to check that
\begin{equation*}
\begin{aligned}
    &\left(\frac{\partial}{\partial t}G(\rho,t,{\mathbf{x_0}})+\lambda G(\rho,t,{\mathbf{x_0}})-(r(\mathbf{x_0})+J\rho U(\mathbf{x_0}))G(\rho,t,{\mathbf{x_0}})\right)e^{(\lambda-r(\mathbf{x_0})-J\rho U(\mathbf{x_0}))t}\\
    &=\frac{\partial}{\partial t}\left(e^{(\lambda-r(\mathbf{x_0})-J\rho U(\mathbf{x_0}))t}G(\rho,t,{\mathbf{x_0}})\right).
\end{aligned}
\end{equation*}
Under the assumption that the solution to Eq. \eqref{equFk} is sufficiently regular, \eqref{equFk2} can be rewritten as
\begin{equation*}
    \begin{aligned}
        ~^{C}_0D_{t}^{\alpha}\left(e^{(\lambda-r(\mathbf{x_0})-J\rho U(\mathbf{x_0}))t}G(\rho,t,{\mathbf{x_0}})\right)=&e^{(\lambda-r(\mathbf{x_0})-J\rho U(\mathbf{x_0}))t} (\Delta+\gamma)^{\frac{\beta}{2}} G(\rho,t,{\mathbf{x_0}})+\lambda ~_0I_{t}^{1-\alpha}e^{\lambda t}\\
        &+\lambda^{\alpha}e^{(\lambda-r(\mathbf{x_0})-J\rho U(\mathbf{x_0}))t}G(\rho,t,{\mathbf{x_0}})-\lambda^{\alpha}e^{\lambda t}.
    \end{aligned}
\end{equation*}
According to Definition \ref{def2}, we have the following equivalent form of Eq. \eqref{equFk}
\begin{equation}\label{equFkequ1}
    \begin{aligned}
    &e^{(r(\mathbf{x_0})+J\rho U(\mathbf{x_0}))t}~^{C}_0D_{t}^{\alpha,\lambda}\left(e^{-(r(\mathbf{x_0})+J\rho U(\mathbf{x_0}))t}G(\rho,t,{\mathbf{x_0}})\right)-\lambda^{\alpha}G(\rho,t,{\mathbf{x_0}})\\
    &=(\Delta+\gamma)^{\frac{\beta}{2}} G(\rho,t,{\mathbf{x_0}})+f,
    \end{aligned}
\end{equation}
where
\begin{equation*}
f=-\lambda^{\alpha} e^{(r(\mathbf{x_0})+J\rho U(\mathbf{x_0}))t}+\lambda e^{-(\lambda-r(\mathbf{x_0})-J\rho U(\mathbf{x_0}))t}~_0I_{t}^{1-\alpha}e^{\lambda t}.
\end{equation*}

So, to get an effective numerical scheme for Eq. \eqref{equFk} with nonhomogeneous initial condition, we need to homogenize the initial condition for Eq. \eqref{equFkequ1}, i.e.,
\begin{equation}\label{equinitial0equ}
  G(\rho,t,{\mathbf{x_0}})=W(\rho,t,{\mathbf{x_0}})+G(\rho,0,{\mathbf{x_0}})e^{-(\lambda-r(\mathbf{x_0})-J\rho U(\mathbf{x_0}))t}.
\end{equation}
So \eqref{equFkequ1} can be rewritten as
\begin{equation}\label{equFkequ2qc}
    \begin{aligned}
    &e^{(r(\mathbf{x_0})+J\rho U(\mathbf{x_0}))t}~^{C}_0D_{t}^{\alpha,\lambda}\left(e^{-(r(\mathbf{x_0})+J\rho U(\mathbf{x_0}))t}W(\rho,t,{\mathbf{x_0}})\right)-\lambda^{\alpha}W(\rho,t,{\mathbf{x_0}})\\
    &=(\Delta+\gamma)^{\frac{\beta}{2}} W(\rho,t,{\mathbf{x_0}})+f_w,
    \end{aligned}
\end{equation}
where
\begin{equation}\label{equinitial0f}
    \begin{aligned}
        f_w=&\lambda^{\alpha}G(\rho,0,{\mathbf{x_0}})e^{-(\lambda-r(\mathbf{x_0})-J\rho U(\mathbf{x_0}))t}+(\Delta+\gamma)^{\frac{\beta}{2}} \left(G(\rho,0,{\mathbf{x_0}})e^{-(\lambda-r(\mathbf{x_0})-J\rho U(\mathbf{x_0}))t}\right)\\
        &-\lambda^{\alpha} e^{(r(\mathbf{x_0})+J\rho U(\mathbf{x_0}))t}+\lambda e^{-(\lambda-r(\mathbf{x_0})-J\rho U(\mathbf{x_0}))t}~_0I_{t}^{1-\alpha}e^{\lambda t}.
    \end{aligned}
\end{equation}
According to \eqref{equFkequ1} and \eqref{equinitial0f}, we can get
\begin{equation*}
f_w(0)=0.
\end{equation*}

  Comparing Eq. \eqref{equFkequ2qc} with Eq. \eqref{equFkequ1}, we only need to consider how to develop the numerical scheme for \eqref{equFkequ1} with homogeneous initial condition, i.e., $G_0(\rho,\mathbf{x_0})=0$, since the numerical scheme can also be applied to  \eqref{equFkequ2qc}-\eqref{equinitial0f}. For convenience, we rewrite \eqref{equFkequ1} as
\begin{equation}\label{equFKtodis}
    \begin{aligned}
\mathcal{L}^{\alpha,\lambda}_t G(\rho,t,{\mathbf{x_0}})=(\Delta+\gamma)^{\frac{\beta}{2}} G(\rho,t,{\mathbf{x_0}})+f,
    \end{aligned}
\end{equation}
where
\begin{equation}\label{equdefL}
  \mathcal{L}^{\alpha,\lambda}_tG(\rho,t,{\mathbf{x_0}})=e^{(r(\mathbf{x_0})+J\rho U(\mathbf{x_0}))t}~^{C}_0D_{t}^{\alpha,\lambda}\left(e^{-(r(\mathbf{x_0})+J\rho U(\mathbf{x_0}))t}G(\rho,t,{\mathbf{x_0}})\right)-\lambda^{\alpha}G(\rho,t,{\mathbf{x_0}})
\end{equation}
and $f$ satisfies
\begin{equation}\label{equFkf00}
f(0)=0.
\end{equation}

\begin{remark}
Through the above derivation, it can be noted that one only needs to discretize Eq. \eqref{equFKtodis} to approximate Eq. \eqref{equFk}.
\end{remark}

\section{Space discretization and error analysis}
This section provides a finite difference discretization for the two-dimensional tempered fractional Laplacian on a bounded domain $\Omega=(-l,l)\times(-l,l)$ with extended homogeneous Dirichlet boundary conditions: $G(x,y)\equiv 0$ for $(x,y)\in\Omega^c$, which is based on our previous work \cite{Sun2018} and modifies the regularity requirement according to \cite{Zhang2017}. Afterwards, we give the error analysis of the space semi-discrete scheme. Here, we set the mesh sizes $h_1=l/N_{i}$ and $h_2=l/N_{j}$; denote grid points $x_i=ih_1$ and $y_j=jh_2$, for $-N_{i}\leq i \leq N_{i} $ and $ -N_{j}\leq j \leq N_{j}$; for convenience, let $N_i=N_j=N$, then we can set $h_1=h_2=h$.
\subsection{Spatial discretization}
According to \eqref{spatemdef}, we have
\begin{equation}\label{spatemdef2}
    \begin{split}
       -(\Delta+\gamma)^{\frac{\beta}{2}}G(x,y)=-c_{2,\beta}{\rm P.V.}\int\int_{\mathbb{R}^2}\frac{G(\xi,\eta)-G(x,y)}{\vartheta(x,y,\xi,\eta)}d\xi d\eta,
    \end{split}
\end{equation}
where
\begin{equation}\label{equdefvar}
\vartheta(x,y,\xi,\eta)=e^{\gamma\sqrt{(\xi-x)^2+(\eta-y)^2}}\left(\sqrt{(\xi-x)^2+(\eta-y)^2}\right)^{{2+\beta}}.
\end{equation}
To discretize $(\Delta+\gamma)^{\frac{\beta}{2}}G(x_p,y_q)$ for any $-N\leq p,\,q\leq N$, we first divide the integral domain into two parts for \eqref{spatemdef2}, i.e.,
\begin{equation}\label{spatemdef3}
    \begin{aligned}
        \int\int_{\mathbb{R}^2}\frac{G(\xi,\eta)-G(x_p,y_q)}{\vartheta(x_p,y_q,\xi,\eta)}d\xi d\eta=&\int\int_{\Omega}\frac{G(\xi,\eta)-G(x_p,y_q)}{\vartheta(x_p,y_q,\xi,\eta)}d\xi d\eta\\
        &+\int\int_{\mathbb{R}^2\backslash\Omega}\frac{G(\xi,\eta)-G(x_p,y_q)}{\vartheta(x_p,y_q,\xi,\eta)}d\xi d\eta.
    \end{aligned}
\end{equation}
It is easy to see that
 \begin{equation*}
    \begin{aligned}
        &\int\int_{\mathbb{R}^2\backslash\Omega}\frac{G(\xi,\eta)-G(x_p,y_q)}{\vartheta(x_p,y_q,\xi,\eta)}d\xi d\eta\\
        =&\int\int_{\mathbb{R}^2\backslash\Omega}\frac{G(\xi,\eta)}{\vartheta(x_p,y_q,\xi,\eta)}d\xi d\eta-\int\int_{\mathbb{R}^2\backslash\Omega}\frac{G(x_p,y_q)}{\vartheta(x_p,y_q,\xi,\eta)}d\xi d\eta\\
        =&-W_{p,q}^\infty G(x_p,y_q),
    \end{aligned}
\end{equation*}
where the fact $G(\xi,\eta)\equiv 0$ for $(\xi,\eta)\in\mathbb{R}^2\backslash \Omega$ is used and
\begin{equation*}
    \begin{aligned}
        W_{p,q}^\infty=\int\int_{\mathbb{R}^2\backslash\Omega}\frac{1}{\vartheta(x_p,y_q,\xi,\eta)}d\xi d\eta.
    \end{aligned}
\end{equation*}
Here, $W_{p,q}^\infty$ can be calculated by the function `integral2.m' in MATLAB.

Next, we formulate the first integral in \eqref{spatemdef3} as
\begin{equation}\label{equtodis}
\int\int_\Omega \frac{G(\xi,\eta)-G(x_p,y_q)}{\vartheta(x_p,y_q,\xi,\eta)} d\eta d\xi =\sum_{i=-N}^{N-1}\sum_{j=-N}^{N-1}\int_{\xi_{i}}^{\xi_{i+1}}\int_{\eta_{j}}^{\eta_{j+1}}\frac{G(\xi,\eta)-G(x_p,y_q)}{\vartheta(x_p,y_q,\xi,\eta)}d\eta d\xi,
\end{equation}
where $\xi_i=ih$ and $\eta_j=jh$.
Denote $\mathcal{I}_{p,q}=\{(p,q),(p-1,q),(p,q-1),(p-1,q-1)\}$. For \eqref{equtodis}, when $(i,j)\in \mathcal{I}_{p,q}$, we rewrite them as
\begin{equation}\label{equsingu1}
    \begin{aligned}
        &\int_{\xi_{p-1}}^{\xi_{p}}\int_{\eta_{q-1}}^{\eta_{q}}\frac{G(\xi,\eta)-G(x_p,y_q)}{\vartheta(x_p,y_q,\xi,\eta)}d\eta d\xi+\int_{\xi_{p}}^{\xi_{p+1}}\int_{\eta_{q-1}}^{\eta_{q}}\frac{G(\xi,\eta)-G(x_p,y_q)}{\vartheta(x_p,y_q,\xi,\eta)}d\eta d\xi\\
        &+\int_{\xi_{p-1}}^{\xi_{p}}\int_{\eta_{q}}^{\eta_{q+1}}\frac{G(\xi,\eta)-G(x_p,y_q)}{\vartheta(x_p,y_q,\xi,\eta)}d\eta d\xi+\int_{\xi_{p}}^{\xi_{p+1}}\int_{\eta_{q}}^{\eta_{q+1}}\frac{G(\xi,\eta)-G(x_p,y_q)}{\vartheta(x_p,y_q,\xi,\eta)}d\eta d\xi\\
        =&\int_{\xi_{p-1}}^{\xi_{p+1}}\int_{\eta_{q-1}}^{\eta_{q+1}}\frac{G(\xi,\eta)-G(x_p,y_q)}{\vartheta(x_p,y_q,\xi,\eta)}d\eta d\xi.
    \end{aligned}
\end{equation}
According to \eqref{equdefvar}, we obtain
\begin{equation}\label{symmetry}
\vartheta(x,y,x-\xi,y-\eta)=\vartheta(x,y,x+\xi,y-\eta)=\vartheta(x,y,x-\xi,y+\eta)=\vartheta(x,y,x+\xi,y+\eta).
\end{equation}
By \eqref{symmetry} and the symmetry of the integral domain and integrand, Eq. \eqref{equsingu1} can be rewritten as
\begin{equation}\label{equsingu2}
    \begin{aligned}
        &\int_{\xi_{p-1}}^{\xi_{p+1}}\int_{\eta_{q-1}}^{\eta_{q+1}}\frac{G(\xi,\eta)-G(x_p,y_q)}{\vartheta(x_p,y_q,\xi,\eta)}d\eta d\xi=\int_{0}^{h}\int_{0}^{h}\frac{\psi(x_p,y_q,\xi,\eta)}{e^{\gamma\sqrt{\xi^2+\eta^2}}\left(\sqrt{\xi^2+\eta^2}\right)^{{2+\beta}}}d\eta d\xi,
    \end{aligned}
\end{equation}
where
\begin{equation*}
\begin{aligned}
\psi(x_p,y_q,\xi,\eta)=&G(x_p+\xi,y_q+\eta)+G(x_p-\xi,y_q+\eta)\\&+G(x_p-\xi,y_q-\eta)+G(x_p+\xi,y_q-\eta)-4G(x_p,y_q).
\end{aligned}
\end{equation*}
Further denoting
\begin{equation*}
\phi_{\sigma}(\xi,\eta)=\frac{\psi(x_p,y_q,\xi,\eta)}{e^{\gamma\sqrt{\xi^2+\eta^2}}\left(\sqrt{\xi^2+\eta^2}\right)^{{\sigma}}}, ~~\sigma\in(\beta,2],
\end{equation*}
then Eq. \eqref{equsingu2} can be written as
\begin{equation}\label{equsingu3}
\begin{aligned}
 &\int_{\xi_{p-1}}^{\xi_{p+1}}\int_{\eta_{q-1}}^{\eta_{q+1}}\frac{G(\xi,\eta)-G(x_p,y_q)}{\vartheta(x_p,y_q,\xi,\eta)}d\eta d\xi
 =\int_{0}^{h}\int_{0}^{h}\phi_{\sigma}(\xi,\eta)(\xi^2+\eta^2)^{\frac{\sigma-2-\beta}{2}}d\eta d\xi.
\end{aligned}
\end{equation}
 Here, we use the weighted trapezoidal rule to approximate \eqref{equsingu3}, that is,
\begin{equation}\label{equdis11}
\begin{split}
\int_{0}^{h}\int_{0}^{h}&\phi_{\sigma}(\xi,\eta)(\xi^2+\eta^2)^{\frac{\sigma-2-\beta}{2}} d\eta d\xi\approx\\&\left\{
    \begin{split}
      &\frac{1}{4}\left(\lim_{(\xi,\eta)\rightarrow(0,0)}\phi_{\sigma}(\xi,\eta)+\phi_{\sigma}(\xi_0,\eta_1)+\phi_{\sigma}(\xi_1,\eta_1)+\phi_{\sigma}(\xi_1,\eta_0)\right)W_{0,0},~~\sigma\in(\beta,2);\\
      &\frac{1}{3}\left(\phi_{\sigma}(\xi_0,\eta_1)+\phi_{\sigma}(\xi_1,\eta_1)+\phi_{\sigma}(\xi_1,\eta_0)\right)W_{0,0},~~~~\sigma=2,
    \end{split}
    \right.
    \end{split}
\end{equation}
where
\begin{equation}\label{equdefG00}
W_{0,0}=\int_{0}^{h}\int_{0}^{h}(\xi^2+\eta^2)^{\frac{\sigma-2-\beta}{2}} d\eta d\xi.
\end{equation}
Assuming that $u$ is smooth enough, for $\sigma \in(\beta,2)$, there exists
\begin{equation*}
  \lim_{(\xi,\eta)\rightarrow(0,0)}\phi_{\sigma}(\xi,\eta)=0;
\end{equation*}
and further introduce a parameter 
\begin{equation*}
  k_\sigma=\left\{
  \begin{split}
  1~~~~~~~~~~~~~~&\sigma\in(\beta,2),\\
  \frac{4}{3}~~~~~~~~~~~~~~&\sigma=2.
  \end{split}
  \right.
\end{equation*}
So, Eq. \eqref{equdis11} can be rewritten as
\begin{equation}\label{eqsingularint}
  \int_{0}^{h}\int_{0}^{h}\phi_{\sigma}(\xi,\eta)(\xi^2+\eta^2)^{\frac{\sigma-2-\beta}{2}}d\eta d\xi\approx\frac{k_\sigma}{4}\left(\phi_{\sigma}(\xi_0,\eta_1)+\phi_{\sigma}(\xi_1,\eta_1)+\phi_{\sigma}(\xi_1,\eta_0)\right)W_{0,0}.
\end{equation}
For \eqref{equtodis}, when $(i,j) \notin \mathcal{I}_{p,q}$, denote $I_{p,q,p+i,q+j}$ as the approximation of
\begin{equation*}
 \int_{\xi_{p+i}}^{\xi_{p+i+1}}\int_{\eta_{q+j}}^{\eta_{q+j+1}}\frac{G(\xi,\eta)-G(x_p,y_q)}{\vartheta(x_p,y_q,\xi,\eta)}d\eta d\xi;
 \end{equation*}
 we use the bilinear interpolation to approximate $\frac{G(\xi,\eta)-G(x_p,y_q)}{e^{\gamma \sqrt{(\xi-x)^2+(\eta-y)^2}}}$ in $[\xi_{p+i},\xi_{p+i+1}]\times [\eta_{q+j},\eta_{q+j+1}]$ and get
\begin{equation*}
  \begin{split}
      I_{p,q,p+i,q+j}=&\left(\frac{G(\xi_{p+i},\eta_{q+j})-G(x_p,y_q)}{e^{\gamma h\sqrt{i^2+j^2}}}\right)\left(H^{\xi\eta}_{i,j}-\xi_{i+1}H^\eta_{i,j}-\eta_{j+1}H^\xi_{i,j}+\xi_{i+1}\eta_{j+1}H_{i,j}\right)\\
                        &-\left(\frac{G(\xi_{p+i+1},\eta_{q+j})-G(x_p,y_q)}{e^{\gamma h\sqrt{(i+1)^2+j^2}}}\right)\left(H^{\xi\eta}_{i,j}-\xi_{i}H^\eta_{i,j}-\eta_{j+1}H^\xi_{i,j}+\xi_{i}\eta_{j+1}H_{i,j}\right)\\
                        &-\left(\frac{G(\xi_{p+i},\eta_{q+j+1})-G(x_p,y_q)}{e^{\gamma h\sqrt{i^2+(j+1)^2}}}\right)\left(H^{\xi\eta}_{i,j}-\xi_{i+1}H^\eta_{i,j}-\eta_{j}H^\xi_{i,j}+\xi_{i+1}\eta_{j}H_{i,j}\right)\\
                        &+\left(\frac{G(\xi_{p+i+1},\eta_{q+j+1})-G(x_p,y_q)}{e^{\gamma h\sqrt{(i+1)^2+(j+1)^2}}}\right)\left(H^{\xi\eta}_{i,j}-\xi_{i}H^\eta_{i,j}-\eta_{j}H^\xi_{i,j}+\xi_{i}\eta_{j}H_{i,j}\right),
  \end{split}
\end{equation*}
where
\begin{equation}\label{equdefG}
\begin{split}
H_{i,j}&=\frac{1}{h^2}\int_{\xi_{i}}^{\xi_{i+1}}\int_{\eta_{j}}^{\eta_{j+1}}(\xi^2+\eta^2)^{\frac{-2-\beta}{2}} d\eta d\xi,\\
H^\xi_{i,j}&=\frac{1}{h^2}\int_{\xi_{i}}^{\xi_{i+1}}\int_{\eta_{j}}^{\eta_{j+1}}\xi(\xi^2+\eta^2)^{\frac{-2-\beta}{2}} d\eta d\xi,\\
H^\eta_{i,j}&=\frac{1}{h^2}\int_{\xi_{i}}^{\xi_{i+1}}\int_{\eta_{j}}^{\eta_{j+1}}\eta(\xi^2+\eta^2)^{\frac{-2-\beta}{2}}d\eta d\xi, \\
H^{\xi\eta}_{i,j}&=\frac{1}{h^2}\int_{\xi_{i}}^{\xi_{i+1}}\int_{\eta_{j}}^{\eta_{j+1}}\xi\eta(\xi^2+\eta^2)^{\frac{-2-\beta}{2}} d\eta d\xi.
\end{split}
\end{equation}
Here, $H_{i,j}$, $H^\xi_{i,j}$, $H^\eta_{i,j}$, $H^{\xi\eta}_{i,j}$ can be obtained by numerical integration. Denote
\begin{equation}\label{equdefW}
  \begin{split}
  W^1_{i,j}&=H^{\xi\eta}_{i,j}-\xi_{i+1}H^\eta_{i,j}-\eta_{j+1}H^\xi_{i,j}+\xi_{i+1}\eta_{j+1}H_{i,j},\\
  W^2_{i,j}&=-\left(H^{\xi\eta}_{i-1,j}-\xi_{i-1}H^\eta_{i-1,j}-\eta_{j+1}H^\xi_{i-1,j}+\xi_{i-1}\eta_{j+1}H_{i-1,j}\right),\\
  W^3_{i,j}&=-\left(H^{\xi\eta}_{i,j-1}-\xi_{i+1}H^\eta_{i,j-1}-\eta_{j-1}H^\xi_{i,j-1}+\xi_{i+1}\eta_{j-1}H_{i,j-1}\right),\\
  W^4_{i,j}&=H^{\xi\eta}_{i-1,j-1}-\xi_{i-1}H^\eta_{i-1,j-1}-\eta_{j-1}H^\xi_{i-1,j-1}+\xi_{i-1}\eta_{j-1}H_{i-1,j-1}.
  \end{split}
\end{equation}
Then, for $(i-p,j-q) \notin \mathcal{I}_{p,q}$, $I_{p,q,i,j}$ can be rewritten as
\begin{equation}\label{equIij2}
  \begin{split}
      I_{p,q,i,j}=&\frac{G(\xi_i,\eta_j)-G(x_p,y_q)}{e^{\gamma h\sqrt{(i-p)^2+(j-q)^2}}}W^1_{i-p,j-q}+\frac{G(\xi_{i+1},\eta_j)-G(x_p,y_q)}{e^{\gamma h\sqrt{(i+1-p)^2+(j-q)^2}}}W^2_{(i+1)-p,j-q}\\
      &+\frac{G(\xi_i,\eta_{j+1})-G(x_p,y_q)}{e^{\gamma h\sqrt{(i-p)^2+(j+1-q)^2}}}W^3_{i -p,(j+1)-q}
      \\
      &
      +\frac{G(\xi_{i+1},\eta_{j+1})-G(x_p,y_q)}{e^{\gamma h\sqrt{(i+1-p)^2+(j+1-q)^2}}}W^4_{(i+1)-p,(j+1)-q};
  \end{split}
\end{equation}
and Eq. \eqref{equtodis} becomes
\begin{equation}\label{equdiswithI}
\begin{split}
&\sum_{i=-N}^{N-1}\sum_{j=-N}^{N-1}\int_{\xi_{i}}^{\xi_{i+1}}\int_{\eta_{j}}^{\eta_{j+1}}\frac{G(\xi,\eta)-G(x_p,y_q)}{\vartheta(x_p,y_q,\xi,\eta)}d\eta d\xi \\
\approx&\frac{k_\sigma}{4}\left(\phi_{\sigma}(\xi_0,\eta_1)+\phi_{\sigma}(\xi_1,\eta_1)+\phi_{\sigma}(\xi_1,\eta_0)\right)W_{0,0}\\
&+\sum^{N-1,N-1}_{
\begin{subarray}{c}
i=-N,j=-N;\\(i,j)\notin \mathcal{I}_{p,q}
\end{subarray}}I_{p,q,i,j}.
\end{split}
\end{equation}
To make the form of weight $w^{\beta,\gamma}_{p,q,i,j}$ unified, according to \eqref{eqsingularint}, we denote
\begin{equation}\label{equdefWsp}
    \begin{split}
        &W^1_{-1,-1}=W^2_{1,-1}=W^3_{-1,1}=W^4_{1,1}=\frac{k_\sigma}{4}\frac{W_{0,0}}{\left(\sqrt{2}h\right)^{{\sigma}}},\\
        &W^1_{-1,0}=W^3_{-1,0}=W^2_{1,0}=W^4_{1,0}=\\
        &W^1_{0,-1}=W^2_{0,-1}=W^3_{0,1}=W^4_{0,1}=\frac{k_\sigma}{4}\frac{W_{0,0}}{\left(h\right)^{{\sigma}}}.
    \end{split}
\end{equation}
So, we have the discretization scheme
\begin{equation}\label{defdis}
  -(\Delta+\gamma)_h^{\frac{\beta}{2}}G_{p,q}=\sum_{i=-N}^{N}\sum_{j=-N}^{N}w^{\beta,\gamma}_{p,q,i,j}G_{i,j},
\end{equation}
where
\begin{equation}\label{equweightoffl}\footnotesize
\begin{aligned}
w^{\beta,\gamma}_{p,q,i,j}=&-c_{2,\beta}
\left\{
\begin{split}
&-\left(\sum_{i=-N+1,j=-N+1}^{i=N-1,j=N-1}\frac{W^1_{i-p,j-q}+W^2_{i-p,j-q}+W^3_{i-p,j-q}+W^4_{i-p,j-q}}{e^{\gamma h\sqrt{(i-p)^2+(j-q)^2}}}\right.&\\
&~~+\frac{W^1_{-N-p,-N-q}}{e^{\gamma h\sqrt{(-N-p)^2+(-N-q)^2}}}+\frac{W^2_{N-p,-N-q}}{e^{\gamma h\sqrt{(N-p)^2+(-N-q)^2}}}&\\
&~~+\frac{W^3_{-N-p,N-q}}{e^{\gamma h\sqrt{(-N-p)^2+(N-q)^2}}}+\frac{W^4_{N-p,N-q}}{e^{\gamma h\sqrt{(N-p)^2+(N-q)^2}}}&\\
&~~+\sum_{i=-N+1}^{N-1}\frac{W^1_{i-p,-N-q}+W^2_{i-p,-N-q}}{e^{\gamma h\sqrt{(i-p)^2+(-N-q)^2}}}&\\
&~~+\sum_{i=-N+1}^{N-1}\frac{W^3_{i-p,N-q}+W^4_{i-p,N-q}}{e^{\gamma h\sqrt{(i-p)^2+(N-q)^2}}}&\\
&~~+\sum_{j=-N+1}^{N-1}\frac{W^1_{-N-p,j-q}+W^3_{-N-p,j-q}}{e^{\gamma h\sqrt{(-N-p)^2+(j-q)^2}}}&\\
&\left.~~+\sum_{j=-N+1}^{N-1}\frac{W^2_{N-p,j-q}+W^4_{N-p,j-q}}{e^{\gamma h\sqrt{(N-p)^2+(j-q)^2}}}+W_{p,q}^{\infty}\right),&i=p,j=q;\\
&\frac{W^1_{i-p,j-q}+W^2_{i-p,j-q}}{e^{\gamma h\sqrt{(i-p)^2+(j-q)^2}}},&-N<i<N,j=-N;\\
&\frac{W^1_{i-p,j-q}+W^3_{i-p,j-q}}{e^{\gamma h\sqrt{(i-p)^2+(j-q)^2}}},&i=-N, -N<j<N;\\
&\frac{W^3_{i-p,j-q}+W^4_{i-p,j-q}}{e^{\gamma h\sqrt{(i-p)^2+(j-q)^2}}},&-N<i<N,j=N;\\
&\frac{W^2_{i-p,j-q}+W^4_{i-p,j-q}}{e^{\gamma h\sqrt{(i-p)^2+(j-q)^2}}},&i=N, -N<j<N;\\
&\frac{W^1_{i-p,j-q}}{e^{\gamma h\sqrt{(i-p)^2+(j-q)^2}}},& i=-N,j=-N;\\
&\frac{W^2_{i-p,j-q}}{e^{\gamma h\sqrt{(i-p)^2+(j-q)^2}}},& i=N,j=-N;\\
&\frac{W^3_{i-p,j-q}}{e^{\gamma h\sqrt{(i-p)^2+(j-q)^2}}},& i=-N,j=N;\\
&\frac{W^4_{i-p,j-q}}{e^{\gamma h\sqrt{(i-p)^2+(j-q)^2}}},&i=N,j=N;\\
&\frac{W^1_{i-p,j-q}+W^2_{i-p,j-q}+W^3_{i-p,j-q}+W^4_{i-p,j-q}}{e^{\gamma h\sqrt{(i-p)^2+(j-q)^2}}},&otherwise.
\end{split}
\right.
\end{aligned}
\end{equation}

\begin{remark}
Here, we discretize the tempered fractional Laplacian satisfying homogeneous Dirichlet boundary conditions, so \eqref{defdis} can be rewritten as
\begin{equation}\label{defdis2}
 -(\Delta+\gamma)_h^{\frac{\beta}{2}}G_{p,q}=\sum_{i=-N+1}^{N-1}\sum_{j=-N+1}^{dN-1}w^{\beta,\gamma}_{p,q,i,j}G_{i,j}.
\end{equation}
\end{remark}

\subsection{Error analysis for the space semi-discrete scheme}
First, we define an operator from a function to a vector
\begin{equation*}
  \begin{aligned}
    \mathcal{V}:f\rightarrow \bf{f},
  \end{aligned}
\end{equation*}
where $f$ denotes a function,
\begin{equation*}
  \textbf{f}=\{f_{-N+1,-N+1},f_{-N+1,-N+2},\ldots,f_{-N+1,N-1},f_{-N+2,-N+1},\ldots, f_{N-1,N-1}\},
\end{equation*}
and $f_{p,q}=f(x_p,y_q)$.

According to \eqref{equFKtodis} and \eqref{defdis2}, the spatially semi-discrete scheme can be written as
\begin{equation}
 \mathcal{L}_t^{\alpha,\lambda}G_{h,p,q}(t)=-
    \sum_{i=-N+1,j=-N+1}^{N-1,N-1}w^{\beta,\gamma}_{p,q,i,j}G_{h,i,j}(t)+f_{p,q}(t)~~~~{\rm for }~ -N<p,q<N,
\end{equation}
where $G_{h,p,q}(t)$ is the numerical solution at $(x_p,y_q)$ of the spatially semi-discrete scheme.
Denoting
\begin{equation*}
  \mathbf{G_h}(t)=\{G_{h,-N+1,-N+1}(t),G_{h,-N+1,-N+2}(t),\ldots,G_{h,-N+1,N-1}(t),\ldots, G_{h,N-1,N-1}(t)\},
\end{equation*}
then the spatially semi-discrete scheme can be rewritten as
\begin{equation}\label{equFkspatiallysemidis}
 \mathcal{L}_t^{\alpha,\lambda}\mathbf{G_{h}}(t)=(\Delta+\gamma)^{\frac{\beta}{2}}_h\mathbf{G_{h}}(t)+\mathcal{V}f(t).
\end{equation}
 According to Proposition \ref{procapLap} and taking the Laplace transform for \eqref{equFKtodis}, we get
\begin{equation*}
  \widetilde{\mathcal{L}}_t^{\alpha,\lambda}\tilde{G}=(\Delta+\gamma)^{\frac{\beta}{2}}\tilde{G}+\widetilde{f},
\end{equation*}
where $\tilde{G}$ and $\widetilde{f}$ denote the Laplace transforms of $G$ and $f$, respectively, and
\begin{equation*}
  \widetilde{\mathcal{L}}_t^{\alpha,\lambda}=\left((z+\lambda-r(\mathbf{x_0})-J\rho U(\mathbf{x_0}))^\alpha-\lambda^\alpha\right)=:\omega(z,\mathbf{x_0})=:\omega.
\end{equation*}
So we obtain
\begin{equation}\label{equLapofFk}
  \tilde{G}=\left(\omega-(\Delta+\gamma)^{\frac{\beta}{2}}\right)^{-1}\widetilde{f}=:\tilde{E}(z)\tilde{f}.
\end{equation}
Similarly, taking the Laplace transform for \eqref{equFkspatiallysemidis}, we get
\begin{equation*}
   \omega_h\mathbf{\tilde{G}_{h}}(t)=(\Delta+\gamma)^{\frac{\beta}{2}}_h\mathbf{\tilde{G}_{h}}(t)+\mathcal{V}\tilde{f},
\end{equation*}
where
\begin{equation*}
  \omega_h\mathbf{\tilde{G}_{h}}={\rm diag}(\mathcal{V}\omega)\mathbf{\tilde{G}_{h}}
\end{equation*}
and `diag' denotes a diagonal matrix formed from its vector argument. The solution of spatially semi-discrete scheme is
 \begin{equation}\label{equLapofSdis}
  \mathbf{\tilde{G}_h}=\left(\omega_h-(\Delta+\gamma)_h^{\frac{\beta}{2}}\right)^{-1}\mathcal{V}\widetilde{f}=:\tilde{E_h}(z)\mathcal{V}\tilde{f}.
\end{equation}

 The solution of \eqref{equFKtodis} may therefore be obtained by the inverse Laplace transform of \eqref{equLapofFk}, with integration along a line parallel to and to the right of the imaginary axis. So, we need to choose a suitable integral contour to get the error estimate between \eqref{equFKtodis} and \eqref{equFkspatiallysemidis}. Let
 \begin{equation}\label{equthetabound}
   \frac{\pi}{2}\leq \theta <\frac{\pi}{2}+\theta_\epsilon,
 \end{equation}
 where
 \begin{equation}\label{equthetaepdef}
   \theta_\epsilon=\inf_{\mathbf{x_0}\in \bar{\Omega}}\arccos\left(\frac{|\rho U(\mathbf{x_0})|}{\sqrt{r(\mathbf{x_0})^2+(\rho U(\mathbf{x_0}))^2}}\right).
 \end{equation}
 \begin{lemma}\label{lemarg}
 For any $z\in\Sigma_\theta$ and $\mathbf{x_0}\in \bar{\Omega}$, $\omega(z,\mathbf{x_0})\in\Sigma_\theta$ holds, where $\Sigma_\theta=\{z\in\mathbb{C}:|\arg z|\leq \theta\}$.
 \end{lemma}
 \begin{proof}
 We first introduce a notation $\Gamma_\theta=\{z\in \mathbb{C}:|\arg z|=\theta\}\bigcup\{0\}$. \eqref{equthetabound} and \eqref{equthetaepdef} yield that for any $z\in\Sigma_\theta$ and $\mathbf{x_0}\in \bar{\Omega}$,
  \begin{equation*}
    (z-r(\mathbf{x_0})-J\rho U(\mathbf{x_0}))\in\Sigma_\theta.
  \end{equation*}
So we just need to prove $((z+\lambda)^\alpha-\lambda^\alpha)\in\Sigma_\theta$ for any $z\in\Sigma_\theta$ to get $\omega(z,\mathbf{x_0})\in\Sigma_\theta$. By simple calculations, there are
  \begin{equation*}
    \left((z+\lambda)^\alpha-\lambda^\alpha\right)=\lambda^\alpha\left(\left(\frac{z}{\lambda}+1\right)^\alpha-1\right).
  \end{equation*}
So we just need to prove $((z+1)^\alpha-1)\in\Sigma_{\theta}$ for any $z\in\Sigma_\theta$, which is equal to prove $(z+1)^\alpha\in\Sigma_{\theta}+1=\{z\in \mathbb{C}: z-1\neq 0,|\arg (z-1)|<\theta\}$. Therefore, we need to prove $|(z+1)^\alpha|\geq|\bar{z}|$, where  $\bar{z}\in\bar{\Gamma}=\{z\in \mathbb{C}:z-1\in{\Gamma_{\theta}}\}$ and $|\arg \bar{z}|=|\arg (z+1)^\alpha|$. Assuming that $|\arg (z+1)|=\bar{\theta}$, we have $|\arg \bar{z}|=\alpha\bar{\theta}$ and $0<\bar{\theta}<\theta$. By simple calculations, we obtain
  \begin{equation*}
    \begin{aligned}
        &|(z+1)^{\alpha}|=|(z+1)|^{\alpha}\geq\left(\frac{\cos{\tilde{\theta}}}{\cos{(\bar{\theta}-\tilde{\theta})}}\right)^\alpha,\\
        &|\bar{z}|=\frac{\cos{\tilde{\theta}}}{\cos{(\alpha\bar{\theta}-\tilde{\theta})}},
    \end{aligned}
  \end{equation*}
  where $\tilde{\theta}=\theta-\pi/2$. Denote
  \begin{equation*}
    F(\bar{\theta})=\frac{1}{|\bar{z}|}\left(\frac{\cos{\tilde{\theta}}}{\cos{(\bar{\theta}-\tilde{\theta})}}\right)^\alpha=(\cos{\tilde{\theta}})^{\alpha-1}\frac{\cos(\alpha\bar{\theta}-\tilde{\theta})}{(\cos(\bar{\theta}-\tilde{\theta}))^\alpha}.
  \end{equation*}
  Then
  \begin{equation*}
    \begin{aligned}
      F'(\bar{\theta})=&\frac{\alpha(\cos{\tilde{\theta}})^{\alpha-1}}{(\cos(\bar{\theta}-\tilde{\theta}))^{\alpha+1}}
      \left(\sin(\tilde{\theta}-\alpha\bar{\theta})\cos(\bar{\theta}-\tilde{\theta})-\sin(\tilde{\theta}-\bar{\theta})\cos(\alpha\bar{\theta}-\tilde{\theta})\right)\\
      =&\frac{\alpha(\cos{\tilde{\theta}})^{\alpha-1}}{(\cos(\bar{\theta}-\tilde{\theta}))^{\alpha+1}}
       \left(\sin(\bar{\theta}-\tilde{\theta})\cos(\alpha\bar{\theta}-\tilde{\theta})-\sin(\alpha\bar{\theta}-\tilde{\theta})\cos(\bar{\theta}-\tilde{\theta})\right).
    \end{aligned}
  \end{equation*}
  Since $0\leq\bar{\theta}<\theta=\pi/2+\tilde{\theta}$, we have $\cos{\tilde{\theta}}>0$ and $\cos(\bar{\theta}-\tilde{\theta})>0$. Then, for  $ F'(\bar{\theta})$, there are the following discussions.
\begin{enumerate}[1)]
  \item When $\bar{\theta}-\tilde{\theta}>\alpha\bar{\theta}-\tilde{\theta}\geq0$, we have
  \begin{equation*}
    \begin{aligned}
      &\sin(\bar{\theta}-\tilde{\theta})>\sin(\alpha\bar{\theta}-\tilde{\theta})\geq0,\\
      &\cos(\alpha\bar{\theta}-\tilde{\theta})>\cos(\bar{\theta}-\tilde{\theta})>0,\\
    \end{aligned}
  \end{equation*}
  so there is
  \begin{equation*}
     F'(\bar{\theta})>0.
  \end{equation*}

 \item When $\bar{\theta}-\tilde{\theta}\geq0>\alpha\bar{\theta}-\tilde{\theta}$, we have
  \begin{equation*}
    \begin{aligned}
      &\sin(\bar{\theta}-\tilde{\theta})\geq0>\sin(\alpha\bar{\theta}-\tilde{\theta})\\
      &\cos(\alpha\bar{\theta}-\tilde{\theta}),\cos(\bar{\theta}-\tilde{\theta})>0,\\
    \end{aligned}
  \end{equation*}
  so it holds that
  \begin{equation*}
     F'(\bar{\theta})>0.
  \end{equation*}

  \item When $0>\bar{\theta}-\tilde{\theta}>\alpha\bar{\theta}-\tilde{\theta}$, we have
  \begin{equation*}
    \begin{aligned}
      &\sin(\tilde{\theta}-\alpha\bar{\theta})>\sin(\tilde{\theta}-\bar{\theta})>0\\
      &\cos(\bar{\theta}-\tilde{\theta})>\cos(\alpha\bar{\theta}-\tilde{\theta})>0,\\
    \end{aligned}
  \end{equation*}
  so there is
  \begin{equation*}
     F'(\bar{\theta})>0.
  \end{equation*}

  \item When $\bar{\theta}-\tilde{\theta}=\alpha\bar{\theta}-\tilde{\theta}$, that is $\bar{\theta}=0$, there is
  \begin{equation*}
    F'(0)=0.
  \end{equation*}
  \end{enumerate}
Finally, it leads to
  \begin{equation*}
    F(\bar{\theta})\geq F(0)=1
  \end{equation*}
  and
  \begin{equation*}
    |(z+1)^\alpha|\geq|\bar{z}|.
  \end{equation*}
 Therefore, we obtain $\omega(z,\mathbf{x_0})\in\Sigma_{\theta}$ for any $z\in\Sigma_\theta$.
 \end{proof}
 \begin{lemma}\label{lemboundcon}
 $\tilde{E}(z)$ is analytic and satisfies $\|\tilde{E}(z)\|_{L^2\rightarrow L^2}\leq M/|z|^{\alpha}$ in $\Sigma_{\theta,\kappa}$, where $\Sigma_{\theta,\kappa}=\{z\in\mathbb{C}:|z|>\kappa,|\arg z|\leq \theta\}$. Similarly, $\tilde{E}_h(z)$ is analytic and satisfies $\|\tilde{E}_h(z)\|_{l^2\rightarrow l^2}\leq M/|z|^{\alpha}$ in $\Sigma_{\theta,\kappa}$.
 \end{lemma}
  \begin{proof}
  First, we prove $\frac{|z-r(\mathbf{x_0})-J\rho U(\mathbf{x_0})|}{|z|}\geq C$ for any $\mathbf{x_0}\in\bar{\Omega}$. Let
  \begin{equation*}
    \begin{aligned}
      &z=c_1e^{J\eta_1},~~~|\eta_1|<\theta;\\
      &r(\mathbf{x_0})+J\rho U(\mathbf{x_0})=c_2e^{J\eta_2},~~~|\eta_2|\geq\theta_\epsilon+\frac{\pi}{2}.
    \end{aligned}
  \end{equation*}
Then we have
  \begin{equation*}
    \begin{aligned}
      \frac{|z-r(\mathbf{x_0})-J\rho U(\mathbf{x_0})|}{|z|}=&\left|\frac{z-r(\mathbf{x_0})-J\rho U(\mathbf{x_0})}{z}\right|\\
      =&\left|\frac{c_1e^{J\eta_1}-c_2e^{J\eta_2}}{c_1e^{J\eta_1}}\right|\\
      =&\left|1-\frac{c_2}{c_1}e^{J(\eta_2-\eta_1)}\right|.
    \end{aligned}
  \end{equation*}
    By using the relationship between the complex point $1$ and the line $z=|z|e^{J(\eta_2-\eta_1)}$ in complex plane, we obtain
  \begin{equation}\label{equapp1bd1}
    \frac{|z-r(\mathbf{x_0})-J\rho U(\mathbf{x_0})|}{|z|}\geq \sin\left(\theta_\epsilon+\frac{\pi}{2}-\theta\right)\geq C,
  \end{equation}
  where $C$ is a positive constant. Let $z_\epsilon=z-r(\mathbf{x_0})-J\rho U(\mathbf{x_0})$. So there exists a constant $c_\epsilon>0$ satisfying $|z_\epsilon|\geq c_\epsilon>0$ and $|\arg z_\epsilon|\leq \theta$. Moreover, by the mean value theorem, we have
  \begin{equation*}
    \begin{aligned}
      \frac{|(\lambda+z_\epsilon)^\alpha-\lambda^\alpha|}{|z_\epsilon|^\alpha}\geq\frac{C|(\lambda+\varrho z_\epsilon)|^{\alpha-1}| z_\epsilon|}{|z_\epsilon|^\alpha}\geq C\frac{|\varrho z_\epsilon|^{1-\alpha}}{|(\lambda+\varrho z_\epsilon)|^{1-\alpha}},
    \end{aligned}
  \end{equation*}
  where $\varrho\in(0,1)$. Since $|z_\epsilon|\geq c_\epsilon>0$, there exists
  \begin{equation}\label{equapp1bd2}
    \begin{aligned}
      \frac{|(\lambda+z_\epsilon)^\alpha-\lambda^\alpha|}{|z_\epsilon|^\alpha}&\geq C.
    \end{aligned}
  \end{equation}
  According to \eqref{equapp1bd1} and \eqref{equapp1bd2}, we arrive at
  \begin{equation}\label{equomegaest}
    \begin{aligned}
        |\omega|=|(\lambda+z_\epsilon)^\alpha-\lambda^\alpha|\geq C|z_\epsilon|^\alpha\geq C|z|^\alpha.
    \end{aligned}
  \end{equation}

Finally, according to \cite{ZhangDeng2017}, we know that  $-(\Delta+\gamma)^{\frac{\beta}{2}}$ with homogeneous Dirichlet boundary conditions is a positive definite and self-adjoint operator in $L^2(\Omega)$. Thus $-(\Delta+\gamma)^{\frac{\beta}{2}}$ generates an analytic semigroup \cite{Acosta2018,Lubich1996}, which implies that for any $\hat{\theta}$ and $M=M_{\hat{\theta}}$, we have the resolvent estimate
  \begin{equation*}
  \left\|\left(zI-(\Delta+\gamma)^{\frac{\beta}{2}}\right)^{-1}\right\|_{L^2\rightarrow L^2}\leq\frac{M}{|z|}~ ~~{\rm for} ~z\in \Sigma_{\hat{\theta}}=\{z\in\mathbb{C}:z\neq 0,|\arg z|<\hat{\theta}\}.
  \end{equation*}
By Lemma \ref{lemarg}, the condition $\Sigma_{\theta,\kappa}\subseteq\Sigma_{\theta}$ and \eqref{equomegaest}, for fixed $\bar{\mathbf{x}}\in \bar{\Omega}$, we obtain
\begin{equation}\label{equlemfixxest}
  \left\|\left(\omega(z,\bar{\mathbf{x}})-(\Delta+\gamma)^{\frac{\beta}{2}}\right)^{-1}\right\|_{L^2\rightarrow L^2}\leq\frac{M}{|z|^\alpha}~~~ {\rm  for} ~~~z\in \Sigma_{\theta,\kappa}.
  \end{equation}
  Let
  \begin{equation}
 \left (\omega(z,\mathbf{x_0})-(\Delta+\gamma)^{\frac{\beta}{2}}\right)G=F.
  \end{equation}
Then we have
  \begin{equation}\label{equlemanyx}
    \left(\omega(z,\bar{\mathbf{x}})-(\Delta+\gamma)^{\frac{\beta}{2}}\right)G=F+(\omega(z,\bar{\mathbf{x}})-\omega(z,\mathbf{x_0}))G.
  \end{equation}
  It is easy to get that
   \begin{equation}\label{equlemanyxest1}
     \begin{split}
       \|F+(\omega(z,\bar{\mathbf{x}})-\omega(z,\mathbf{x_0}))G\|_{L_2}&\leq \|F\|_{L_2}+\|(\omega(z,\bar{\mathbf{x}})-\omega(z,\mathbf{x_0}))G\|_{L_2}\\
       &\leq \|F\|_{L_2}+\sup_{\mathbf{x_0}\in \bar{\Omega}}|(r(\mathbf{x_0})-r(\bar{\mathbf{x}}))-J\rho(U(\mathbf{x_0})-U(\bar{\mathbf{x}}))|^{\alpha}\|G\|_{L_2}\\
       &\leq \|F\|_{L_2}+\|G\|_{L_2}.\\
     \end{split}
   \end{equation}
  Combining \eqref{equlemfixxest}, \eqref{equlemanyx} and \eqref{equlemanyxest1}, we have
  \begin{equation*}
  \|G\|_{L_2}\leq  M|z|^{-\alpha}\|F\|_{L_2}+M|z|^{-\alpha}\|G\|_{L_2}.
  \end{equation*}
  When $\kappa$ is large enough, namely, $|z|$ is large enough, we have $\|G\|_{L_2}\leq  M|z|^{-\alpha}\|F\|_{L_2}$ and
  \begin{equation*}
  \left\|\left(\omega(z,\mathbf{x_0})-(\Delta+\gamma)^{\frac{\beta}{2}}\right)^{-1}\right\|_{L^2\rightarrow L^2}\leq\frac{M}{|z|^\alpha}~~~ {\rm for} ~~~z\in \Sigma_{\theta,\kappa}.
  \end{equation*}
  Next, to prove $\|\tilde{E}_h(z)\|_{l^2\rightarrow l^2}\leq M/|z|^{\alpha}$, we need to show the positive definiteness and self-adjoint of $(\Delta+\gamma)_h^{\frac{\beta}{2}}$, i.e.,  the matrix $\mathbf{A_s}$ generated by discretizing the tempered fractional Laplacian is positive definite and symmetric there, which is proved in Appendix \ref{appendApos}. The rest of the proof of $\tilde{E_h}(z)$ is similar to the case of $\tilde{E}(z)$. 
  \end{proof}

  By simple calculation, we get the following lemma.
  \begin{lemma}\label{lemDeri}
    $\|\tilde{E}'(z)\|_{L^2\rightarrow L^2}\leq M/|z|^{\alpha+1}$ in $\Sigma_{\theta,\kappa}$, similarly, $\|\tilde{E}_h'(z)\|_{l^2\rightarrow l^2}\leq M/|z|^{\alpha+1}$ in $\Sigma_{\theta,\kappa}$.
  \end{lemma}

\begin{theorem}\label{thmtrunerror}
Denote $(\Delta+\gamma)^{\frac{\beta}{2}}_{h}$ as a finite difference approximation of the tempered fractional Laplacian $(\Delta+\gamma)^{\frac{\beta}{2}}$.
Suppose that $G(x,y)\in C^{2}(\bar{\Omega})$ is supported in an open set $\Omega\subset\mathbb{R}^2$. Then, there are
\begin{equation*}
\begin{aligned}
&\left\|\mathcal{V}(\Delta+\gamma)^{\frac{\beta}{2}}G(x,y)-(\Delta+\gamma)_{h}^{\frac{\beta}{2}}\mathcal{V}G\right\|_{l_\infty}\leq Ch^{2-\beta},\\
&\left\|\mathcal{V}(\Delta+\gamma)^{\frac{\beta}{2}}G(x,y)-(\Delta+\gamma)_{h}^{\frac{\beta}{2}}\mathcal{V}G\right\|_{l_2}\leq Ch^{2-\beta},~~~~~{\rm for}~\beta\in (0,2)
\end{aligned}
\end{equation*}
with $C$ being a positive constant depending on $\beta$ and $\gamma$.
\end{theorem}
\begin{proof}
The details of the proof are given in Appendix A by modifying the proof in \cite{Sun2018}.	
\end{proof}

Before giving the error estimate between \eqref{equFKtodis} and \eqref{equFkspatiallysemidis}, we introduce the following lemma.
\begin{lemma}\label{lemspaerrorLap}
  For $z\in \Sigma_{\theta}$ and the solution $G(\cdot,\mathbf{x})\in C^2(\bar{\Omega})$, there is the estimate
  \begin{equation*}
    \|\mathcal{V}\tilde{E}(z)-\tilde{E}_h(z)\mathcal{V}\|_{L^2\rightarrow l^2}\leq Ch^{2-\beta}|z|^{-\alpha}.
  \end{equation*}
\end{lemma}
\begin{proof}
  According to \eqref{equLapofFk} and \eqref{equLapofSdis}, there are
  \begin{equation}\label{equopeest3}
  \left(\omega-(\Delta+\gamma)^{\frac{\beta}{2}}\right)\tilde{G}=\tilde{f}
  \end{equation}
  and
  \begin{equation}\label{equopeest4}
  \left(\omega_h-(\Delta+\gamma)^{\frac{\beta}{2}}_h\right)\mathbf{\tilde{G}_h}=\mathcal{V}\tilde{f}.
  \end{equation}
Performing $\mathcal{V}$ on both sides of \eqref{equopeest3} leads to
  \begin{equation}\label{equopeest5}
    \mathcal{V}\left(\omega-(\Delta+\gamma)^{\frac{\beta}{2}}\right)\tilde{G}=\mathcal{V}\tilde{f}.
  \end{equation}
  Subtracting \eqref{equopeest5} from \eqref{equopeest4} results in
  \begin{equation*}
        \left(\omega_h -(\Delta+\gamma)^{\frac{\beta}{2}}_h\right)\mathbf{\tilde{G}_h}-\mathcal{V}\left(\omega-(\Delta+\gamma)^{\frac{\beta}{2}}\right)\tilde{G}=0.
  \end{equation*}
  Then we obtain
  \begin{equation*}
  \begin{array}{l}
     \left(\omega_h-(\Delta+\gamma)^{\frac{\beta}{2}}_h\right)\mathbf{\tilde{G}_h}-\left(\omega_h -(\Delta+\gamma)^{\frac{\beta}{2}}_h\right)\mathcal{V}\tilde{G}
     \\
     +\left(\omega_h -(\Delta+\gamma)^{\frac{\beta}{2}}_h\right)\mathcal{V}\tilde{G}-\mathcal{V}\left(\omega -(\Delta+\gamma)^{\frac{\beta}{2}}\right)\tilde{G}=0.
     \end{array}
  \end{equation*}
Rearranging the terms leads to
  \begin{equation}\label{equlemsppf1}
    \left(\omega_h -(\Delta+\gamma)^{\frac{\beta}{2}}_h\right)(\mathbf{\tilde{G}_h}-\mathcal{V}\tilde{G})=(\Delta+\gamma)^{\frac{\beta}{2}}_{h}\mathcal{V}\tilde{G}-\mathcal{V}(\Delta+\gamma)^{\frac{\beta}{2}}\tilde{G},
  \end{equation}
  where the fact $\omega_h \mathcal{V}\tilde{G}=\mathcal{V}\omega \tilde{G}$ is used. According to Theorem \ref{thmtrunerror}, Lemma \ref{lemboundcon} and \eqref{equlemsppf1}, we get
  \begin{equation}
    \begin{aligned}
        \|\mathcal{V}\tilde{G}-\mathbf{\tilde{G}_h}\|_{l^2}&=\left\|\left(\omega_h -(\Delta+\gamma)^{\frac{\beta}{2}}_h\right)^{-1}\left((\Delta+\gamma)^{\frac{\beta}{2}}_{h}\mathcal{V}\tilde{G}-\mathcal{V}(\Delta+\gamma)^{\frac{\beta}{2}}\tilde{G}\right)\right\|_{l^2}\\
        &\leq \left\|\left(\omega_h -(\Delta+\gamma)^{\frac{\beta}{2}}_h\right)^{-1}\right\|_{l^2\rightarrow l^2}\left\|\left((\Delta+\gamma)^{\frac{\beta}{2}}_{h}\mathcal{V}\tilde{G}-\mathcal{V}(\Delta+\gamma)^{\frac{\beta}{2}}\tilde{G}\right)\right\|_{l^2}\\
        &\leq Ch^{2-\beta}|z|^{-\alpha}.
    \end{aligned}
  \end{equation}
  Combining \eqref{equLapofFk} and \eqref{equLapofSdis} leads to
  \begin{equation*}
    \|\mathcal{V}\tilde{E}(z)\tilde{f}-\tilde{E}_h(z)\mathcal{V}\tilde{f}\|_{l^2}\leq Ch^{2-\beta}|z|^{-\alpha}.
  \end{equation*}
  Taking $\|\tilde{f}\|_{L^2}=1$ results in
   \begin{equation*}
    \|\mathcal{V}\tilde{E}(z)-\tilde{E}_h(z)\mathcal{V}\|_{L^2\rightarrow l^2}\leq Ch^{2-\beta}|z|^{-\alpha},
  \end{equation*}
  which completes the proof.
\end{proof}

According to \eqref{equLapofFk}, \eqref{equLapofSdis} and Lemma \ref{lemboundcon}, we have
 \begin{equation}\label{equorgsollap}
   G(t)=\frac{1}{2\pi J}\int_{\Gamma_{\theta,\kappa}}e^{zt}\tilde{E}(z)\tilde{f}dz
 \end{equation}
 and
 \begin{equation}\label{equspasemsollap}
   \mathbf{G_h}(t)=\frac{1}{2\pi J}\int_{\Gamma_{\theta,\kappa}}e^{zt}\tilde{E}_h(z)\mathcal{V}\tilde{f}dz,
 \end{equation}
 where $\Gamma_{\theta,\kappa}=\{z\in \mathbb{C}:|z|\geq \kappa,|\arg z|=\theta\}\bigcup\{z\in \mathbb{C}:|z|=\kappa,|\arg z|\leq\theta\}$.

\begin{theorem}\label{thmspatialEr}
  Let $G$ be the solution of Eq. \eqref{equFKtodis} satisfying $G(\cdot,x)\in C^2(\bar{\Omega})$  and $\mathbf{G_h}$ be the solution of Eq. \eqref{equFkspatiallysemidis}. Suppose $f\in L^\infty(0,T,L^2(\Omega))$ with $\int_{0}^t(t-s)^{\alpha-1}\|f(s)\|_{L^2}ds\leq\infty$. Then, we have the following error estimate between the solutions of \eqref{equFKtodis} and \eqref{equFkspatiallysemidis}
  \begin{equation*}
    \|\mathcal{V}G(t)-\mathbf{G_h}(t)\|_{l^2}\leq Ch^{2-\beta}\int_{0}^t(t-s)^{\alpha-1} \|f(s)\|_{L^2(\Omega)}ds.
  \end{equation*}
\end{theorem}
\begin{proof}
  According to \eqref{equorgsollap} and \eqref{equspasemsollap}, we have
  \begin{equation*}
  \mathcal{V}G(t)-\mathbf{G_h}(t)=\frac{1}{2\pi J}\int_{\Gamma_{\theta,\kappa}}e^{zt}\left(\mathcal{V}\tilde{E}(z)-\tilde{E}_h(z)\mathcal{V}\right)\tilde{f}dz.
  \end{equation*}
From Lemma \ref{lemspaerrorLap} and the property of convolution, there exists
  \begin{equation*}
    \begin{aligned}
      \|\mathcal{V}G(t)-\mathbf{G_h}(t)\|_{l^2}&\leq\left\|\frac{1}{2\pi J}\int_{\Gamma_{\theta,\kappa}}e^{zt}\left(\mathcal{V}\tilde{E}(z)-\tilde{E}(z)_h\mathcal{V}\right)dz\ast f\right\|_{l^2}\\
      &\leq\left\|\frac{1}{2\pi J}\int_{\Gamma_{\theta,\kappa}}e^{zt}\left(\mathcal{V}\tilde{E}(z)-\tilde{E}_h(z)\mathcal{V}\right)dz\right\|_{L^2\rightarrow l^2}\ast \|f\|_{L^2(\Omega)}\\
      &\leq Ch^{2-\beta}\int_{\Gamma_{\theta,\kappa}}e^{-sin(\theta-\pi/2)|z|t}|z|^{-\alpha}d|z|\ast \|f\|_{L^2(\Omega)}\\
      &\leq Ch^{2-\beta}t^{\alpha-1}\ast \|f\|_{L^2(\Omega)},
    \end{aligned}
  \end{equation*}
  which completes the proof.
\end{proof}

\section{Time discretization and error analysis}

In this section, we use the backward Euler (BE) method and the second-order backward difference (SBD) method to discretize the time tempered fractional substantial derivative and obtain the first-order and second-order schemes. Following that, we perform the error analyses for these two schemes.

\subsection{BE scheme and error analysis}

First, let the time step size $\tau=T/L$, $L\in\mathbb{N}$, $t_i=i\tau$, $i=0,1,\ldots,L$ and $0=t_0<t_1<\cdots<t_L=T$. Taking $\delta(\zeta)=(1-\zeta)$ and using convolution quadrature, for \eqref{equdefL} we have the discretization scheme
\begin{equation}\label{equBEdisO1}
    \mathcal{L}^{\alpha,\lambda}_tv(\mathbf{x_0},t_{n})\approx\frac{1}{\tau^\alpha}\sum_{j=0}^n d^{\alpha,\lambda}_{n-j}(\mathbf{x_0})v_j(\mathbf{x_0}),
\end{equation}
where
\begin{equation*}
  v_j(\mathbf{x_0})=v(\mathbf{x_0},t_j)
\end{equation*}
and
\begin{equation}\label{equO1genw1}
    \sum_{j=0}^\infty d^{\alpha,\lambda}_{j}(\mathbf{x_0})\zeta^j=\left(1-\zeta+\tau\lambda-\tau r(\mathbf{x_0})-\tau J\rho U(\mathbf{x_0})\right)^\alpha-(\tau\lambda)^\alpha.
\end{equation}
Here
\begin{equation*}
    \begin{aligned}
    d^{\alpha,\lambda}_{j}(\mathbf{x_0})=\left\{
        \begin{aligned}
            &(1+\tau\lambda-\tau r(\mathbf{x_0})-\tau J\rho U(\mathbf{x_0}))^\alpha-(\tau\lambda)^\alpha,&j=0;\\
            &-\alpha(1+\tau\lambda-\tau r(\mathbf{x_0})-\tau J\rho U(\mathbf{x_0}))^{\alpha-1},&j=1;\\
            &-\frac{(\alpha-j+1)d^{\alpha,\lambda}_{j-1}(\mathbf{x_0})}{j(1+\tau\lambda-\tau r(\mathbf{x_0})-\tau J\rho U(\mathbf{x_0}))},&j>1.
        \end{aligned}
    \right.
    \end{aligned}
\end{equation*}
\begin{remark}
By simple calculations, $\Re{d^{\alpha,\lambda}_{0}(\mathbf{x_0})}>0$ holds, where $\Re{z}$ denotes the real part of $z$. Here we denote $d^{\alpha,\lambda}_{j}(x_p,y_q)$ as $d^{\alpha,\lambda}_{j,p,q}$. The coefficients $d^{\alpha,\lambda}_{j}(\mathbf{x_0})$ can also be calculated by Fast Fourier Transform (FFT).
\end{remark}
Then the time semi-discrete scheme is as follows,
\begin{equation}\label{equO1ltimesemidis}
\frac{1}{\tau^\alpha}\sum_{k=0}^n d^{\alpha,\lambda}_{k}G(\rho,t_{n-k})=
    (\Delta+\gamma)^{\frac{\beta}{2}} G(\rho,t_{n})+f(\rho,t_{n}).
\end{equation}
Combining \eqref{equO1ltimesemidis} with \eqref{defdis}, we obtain the fully discrete scheme of Eq. \eqref{equFKtodis}, i.e., BE scheme
\begin{equation}\label{equO1fulldis}
    \frac{1}{\tau^\alpha}\sum_{k=0}^n {\rm diag}(\mathcal{V}d^{\alpha,\lambda}_{k})\mathbf{G_h^{n-k}}=(\Delta+\gamma)^{\frac{\beta}{2}}_h\mathbf{G_h^n}+\mathcal{V}f^n,
\end{equation}
where $f^n=f(t_n)$, $G^{n}_{h,p,q}$ is the numerical solution at $(x_p,y_q,t_n)$ for fully discrete scheme and
\begin{equation*}
  \mathbf{G_h^{n}}=(G^{n}_{h,-N+1,-N+1},G^{n}_{h,-N+1,-N+2},\ldots,G^{n}_{h,N-1,N-1}).
\end{equation*}
\begin{theorem}\label{thmBEsemER}
	Let $\mathbf{G_h}$ and $\mathbf{G^n_h}$ be the solutions of Eq. \eqref{equFkspatiallysemidis} and Eq. \eqref{equO1fulldis}, respectively. If $f\in L^{\infty}(0,T,L^2(\Omega))$ with $\int_{0}^t(t-s)^{\alpha-1}\|\mathcal{V}f'(s)\|_{l^2}ds\leq\infty$ for $t\in(0,T]$, then we have
	\begin{equation*}
	\|\mathbf{G_h}(t_n)-\mathbf{G^n_h}\|_{l^2}\leq C\tau\int_0^{t_n}(t_n-s)^{\alpha-1}\|\mathcal{V}f'(s)\|_{l^2}ds.
	\end{equation*}
\end{theorem}
\begin{proof}
	According to \eqref{equspasemsollap}, we have
	\begin{equation}\label{equO1exactsol}
	\mathbf{G_h}(t_n)=\left(\frac{1}{2\pi J}\int_{\Gamma_{\theta,\kappa}}e^{zt}\tilde{E}_h(z)z^{-1}dz\ast \mathcal{V}f'(t)\right)(t_n),\\
	\end{equation}
	where the fact $f(t)=f(0)+1\ast f'(t)$, the property of convolution and \eqref{equFkf00} are used.
 To get the solution of \eqref{equO1fulldis}, we need to multiply by $\zeta^n$ and sum from $0$ to $\infty$, so
\begin{equation*}
  \sum_{n=0}^{\infty}\left(\frac{1}{\tau^\alpha}\sum_{k=0}^n {\rm diag}(\mathcal{V}d^{\alpha,\lambda}_{k})\mathbf{G_h^{n-k}}\right)\zeta^n=\sum_{n=0}^{\infty}\left((\Delta+\gamma)^{\frac{\beta}{2}}_h\mathbf{G_h^n}\zeta^n+\mathcal{V}f^n\zeta^n\right).
\end{equation*}
According to \eqref{equO1genw1}, we obtain
\begin{equation*}
        \left(\omega_h\left(\frac{1-\zeta}{\tau}\right)-(\Delta+\gamma)^{\frac{\beta}{2}}_h\right)\sum_{n=0}^{\infty}\zeta^n\mathbf{G_h^n}=\sum_{n=0}^{\infty}\zeta^n\mathcal{V}f^n,
\end{equation*}
which implies
\begin{equation*}
\sum_{n=0}^{\infty}\zeta^n\mathbf{G_h^n}= \left(\omega_h\left(\frac{1-\zeta}{\tau}\right)-(\Delta+\gamma)^{\frac{\beta}{2}}_h\right)^{-1}\sum_{n=0}^{\infty}\zeta^n\mathcal{V}f^n=\tilde{E}_h\left (\frac{1-\zeta}{\tau}\right )\sum_{n=0}^{\infty}\zeta^n\mathcal{V}f^n.
\end{equation*}
Thus
\begin{equation*}
	\mathbf{G_h^{n}}=\tau\sum_{j=0}^{n}E_{n-j}\mathcal{V}f^j,
\end{equation*}
where
\begin{equation}\label{equdefofejo1}
	\tilde{E}_h\left (\frac{1-\zeta}{\tau}\right )=\tau\sum_{j=0}^{\infty}E_j\zeta^j.
\end{equation}
For convenience, we denote
\begin{equation*}
	Q(t',v)=\tau\cdot {\rm diag} \left(\sum_{0\leq t_j\leq t'}E_j\mathcal{V}v(t'-t_j)\right)={\rm diag} \left((\hat{E}_\tau\ast \mathcal{V}v)(t')\right),
\end{equation*}
 where $\hat{E}_\tau=\tau\sum\limits_{j=0}^{\infty}E_j\delta_{t_j}$, with $\delta_t$ the delta function concentrated at $ t $. By the fact $f(t)=f(0)+1\ast f'(t)$, the property of convolution and \eqref{equFkf00}, we have
 \begin{equation}\label{equO1numersol}
 	\mathbf{G_h^{n}}=(\hat{E}_\tau\ast\mathcal{V}f)(t_n)={\rm diag}(\hat{E}_\tau\ast 1)\ast(\mathcal{V}f')(t_n)=(Q(t',1)\ast\mathcal{V}f'(t'))(t_n).
 \end{equation}
 According to \eqref{equO1exactsol}, \eqref{equO1numersol} and the property of convolution, we have
 \begin{equation}
 \begin{aligned}
 \mathbf{G_h}(t_n)-\mathbf{G_h^n}=\left(\left(\frac{1}{2\pi J}\int_{\Gamma_{\theta,\kappa}}e^{zt'}\tilde{E}_h(z)z^{-1}dz-Q(t',1)\right)\ast \mathcal{V}f'(t')\right)(t_n).
 \end{aligned}
 \end{equation}

To get the desired bound, we need to consider the error between $\frac{1}{2\pi J}\int_{\Gamma_{\theta,\kappa}}e^{zt'}\tilde{E}_h(z)z^{-1}dz$ and $Q(t',1)$ for $ t'\in[t_{n-1},t_n),~n\geq 1$. As for $n=1$, we have
\begin{equation*}
\left \|Q(t',1)-\frac{1}{2\pi J}\int_{\Gamma_{\theta,\kappa}}e^{zt'}\tilde{E}_h(z)z^{-1}dz\right \|_{l_2\rightarrow\l_2}\leq Ct^{\prime(\alpha-1)}\tau.
\end{equation*}
Then for $n>1$, we have the estimate
 \begin{equation}\label{equO1ERpart1}
 	\begin{split}
 		&\left \|\frac{1}{2\pi J}\int_{\Gamma_{\theta,\kappa}}e^{zt'}\tilde{E}_h(z)z^{-1}dz-\frac{1}{2\pi J}\int_{\Gamma_{\theta,\kappa}}e^{zt_{n-1}}\tilde{E}_h(z)z^{-1}dz\right\|_{l_2\rightarrow\l_2}\\
 		\leq&C\tau\int_{\Gamma_{\theta,\kappa}}e^{-c|z|t'}|z|^{-\alpha}d|z|\\
 		\leq&C\tau  t'^{(\alpha-1)}.
 	\end{split}
\end{equation}
Following the above, we need to prove that
 \begin{equation}\label{equO1ERpart2}
 \left\|	Q(t',1)-\frac{1}{2\pi J}\int_{\Gamma_{\theta,\kappa}}e^{zt_{n-1}}\tilde{E}_h(z)z^{-1}dz\right\|_{l_2 \rightarrow l_2} \leq C\tau t'^{(\alpha-1)}.
 \end{equation}
 According to \eqref{equdefofejo1}, we have for small $\xi_\tau=e^{-\tau(\kappa+1)}$,
 \begin{equation*}
 	\tau E_n=\frac{1}{2\pi J}\int_{|\zeta|=\xi_\tau}\zeta^{-n-1}\tilde{E}_h\left(\frac{1-\zeta}{\tau}\right)d\zeta.
 \end{equation*}
 Since  $\sum\limits_{j=0}^{n-1}\zeta^{-j-1}=(\zeta^{-n}-1)/(1-\zeta)$, $ Q(t',1) $ can be written as
 \begin{equation*}
 	Q(t',1)=\tau\sum_{j=0}^{n-1}E_j=\frac{1}{2\pi \tau J}\int_{|\zeta|=\xi_\tau}\zeta^{-n}\tilde{E}_h\left(\frac{1-\zeta}{\tau}\right)\left(\frac{1-\zeta}{\tau}\right)^{-1}d\zeta.
 \end{equation*}
 where the fact $ \tilde{E}_h\left((1-\zeta)/\tau\right)/(1-\zeta) $ is analytic for small $\zeta$ is used. Taking $\zeta=e^{-z\tau}$, we get
 \begin{equation*}
 Q(t',1)=\frac{1}{2\pi J}\int_{\Gamma^\tau}e^{zt_{n-1}}\tilde{E}_h\left(\frac{1-e^{-z\tau}}{\tau}\right)\left(\frac{1-e^{-z\tau}}{\tau}\right)^{-1}dz,
 \end{equation*}
 where $\Gamma^\tau=\{z=\kappa+1+iy:y\in\mathbb{R}~{\rm and}~|y|\leq \pi/\tau\}$. Next we deform the contour $\Gamma^\tau$ to
 $\Gamma^\tau_{\theta,\kappa}=\{z\in \mathbb{C}:\kappa\leq |z|\leq\frac{\pi}{\tau\sin(\theta)},|\arg z|=\theta\}\bigcup\{z\in \mathbb{C}:|z|=\kappa,|\arg z|\leq\theta\}$, then $Q(t',1)$ can be rewritten as
  \begin{equation*}
 Q(t',1)=\frac{1}{2\pi J}\int_{\Gamma^\tau_{\theta,\kappa}}e^{zt_{n-1}}\tilde{E}_h\left(\frac{1-e^{-z\tau}}{\tau}\right)\left(\frac{1-e^{-z\tau}}{\tau}\right)^{-1}dz.
 \end{equation*}
 Thus we have
 \begin{equation*}
 	\begin{split}
 			&\frac{1}{2\pi J}\int_{\Gamma_{\theta,\kappa}}e^{zt_{n-1}}\tilde{E}_h(z)z^{-1}dz-Q(t',1)\\
 			=&\frac{1}{2\pi J}\int_{\Gamma_{\theta,\kappa}\backslash\Gamma^\tau_{\theta,\kappa}}e^{zt_{n-1}}\tilde{E}_h(z)z^{-1}dz\\
 			&+\frac{1}{2\pi J}\int_{\Gamma^\tau_{\theta,\kappa}}e^{zt_{n-1}}\left (\tilde{E}_h(z)z^{-1}-\tilde{E}_h\left(\frac{1-e^{-z\tau}}{\tau}\right)\left(\frac{1-e^{-z\tau}}{\tau}\right)^{-1}\right )dz\\
 			=&\uppercase\expandafter{\romannumeral1}+\uppercase\expandafter{\romannumeral2}.
 	\end{split}
 \end{equation*}
 For $\uppercase\expandafter{\romannumeral1}$, according to Lemma \ref{lemboundcon} there exists the estimate
\begin{equation}
\|\uppercase\expandafter{\romannumeral1}\|_{l_2\rightarrow l_2}\leq C\int_{\Gamma_{\theta,\kappa}\backslash\Gamma^\tau_{\theta,\kappa}}e^{-c|z|t'}|z|^{-\alpha-1}d|z|\leq C\tau t'^{(\alpha-1)}.
\end{equation}
For $\uppercase\expandafter{\romannumeral2}$, we obtain, with the mean value theorem, Lemma \ref{lemDeri} and $\frac{1-e^{-z\tau}}{\tau}=z+O(\tau z^2)$,
\begin{equation}
  \left\|\tilde{E}_h(z)z^{-1}-\tilde{E}_h\left(\frac{1-e^{-z\tau}}{\tau}\right)\left(\frac{1-e^{-z\tau}}{\tau}\right)^{-1}\right\|_{l_2\rightarrow l_2}\leq C|z|^{-\alpha-2}|\tau z^2|\leq C\tau|z|^{-\alpha}.
\end{equation}
Consequently,
\begin{equation*}
\begin{split}
  \|\uppercase\expandafter{\romannumeral2}\|_{l_2\rightarrow l_2}\leq& C\tau\left|\int_\kappa^{\frac{\pi}{\tau\sin(\theta)}}e^{-crt'}|r|^{-\alpha}dr\right|+C\tau\left|\int_{-\theta}^{\theta}e^{\kappa\cos(\eta) t_{n-1}}\kappa^{1-\alpha}d\eta\right|\\
  \leq& Ct'^{(\alpha-1)}\tau+C\tau\kappa e^{\kappa t'}\leq Ct'^{(\alpha-1)}\tau,
\end{split}
\end{equation*}
where the fact $\kappa e^{\kappa t'}\leq \kappa T^{1-\alpha}e^{\kappa T}t'^{(\alpha-1)}$ is used. Combining \eqref{equO1ERpart1} and \eqref{equO1ERpart2} leads to the desired result.
\end{proof}

\begin{theorem}
  Let $G$ be the solution of Eq. \eqref{equFKtodis} satisfying $G(\cdot,x)\in C^2(\bar{\Omega})$ and $\mathbf{G^n_h}$ be the solution of Eq. \eqref{equO1fulldis}. If $f\in L^{\infty}(0,T,L^2(\Omega))$ satisfying
   \begin{equation*}
    \int_{0}^t(t-s)^{\alpha-1}\|f(s)\|_{L^2}ds\leq\infty
  \end{equation*}
  and
  \begin{equation*}
    \int_{0}^t(t-s)^{\alpha-1}\|\mathcal{V}f'(s)\|_{l^2}ds\leq\infty
  \end{equation*}
  for $t\in(0,T]$, then
  \begin{equation*}
    \|\mathcal{V}G(t_n)-\mathbf{G^n_h}\|_{l^2}\leq C\tau\int_0^{t_n}(t_n-s)^{\alpha-1}\|\mathcal{V}f'(s)\|_{l^2}ds+Ch^{2-\beta}\int_0^{t_n}(t_n-s)^{\alpha-1}\|f(s)\|_{L^2}ds.
  \end{equation*}
\end{theorem}
\begin{proof}
Combining Theorem \ref{thmspatialEr} and Theorem \ref{thmBEsemER} results in
\begin{equation*}
  \begin{aligned}
    \|\mathcal{V}G(t_n)-\mathbf{G^n_h}\|_{l^2}=&\|\mathcal{V}G(t_n)-\mathbf{G_h}(t_n)+\mathbf{G_h}(t_n)-\mathbf{G^n_h}\|_{l^2}\\
        \leq&\|\mathcal{V}G(t_n)-\mathbf{G_h}(t_n)\|_{l^2}+\|\mathbf{G_h}(t_n)-\mathbf{G^n_h}\|_{l^2}\\
        \leq& C\tau\int_0^{t_n}(t_n-s)^{\alpha-1}\|\mathcal{V}f'(s)\|_{l^2}ds+Ch^{2-\beta}\int_0^{t_n}(t_n-s)^{\alpha-1}\|f(s)\|_{L^2}ds.
  \end{aligned}
\end{equation*}
\end{proof}
\subsection{SBD scheme and error analysis }
Taking $\delta(\zeta)=(1-\zeta)+(1-\zeta)^2/2$, there exist
\begin{equation}\label{equBEdisO2}
    \mathcal{L}^{\alpha,\lambda}_tv(\mathbf{x_0},t_n)\approx\frac{1}{\tau^\alpha}\sum_{j=0}^n d^{\alpha,\lambda}_{n-j}(\mathbf{x_0})v_j(\mathbf{x_0})
\end{equation}
and
\begin{equation}\label{equO2genw1}
    \sum_{j=0}^\infty d^{\alpha,\lambda}_{j}(\mathbf{x_0})\zeta^j=((1-\zeta)+(1-\zeta)^2/2+\tau\lambda-\tau r(\mathbf{x_0})-\tau J\rho U(\mathbf{x_0}))^\alpha-(\tau\lambda)^\alpha,
\end{equation}
where $d^{\alpha,\lambda}_{j}(\mathbf{x_0})$ can also be calculated by FFT. As for $d^{\alpha,\lambda}_{0}(\mathbf{x_0})$, we have $\Re{d^{\alpha,\lambda}_{0}(\mathbf{x_0})}>0$. Then the time semi-discrete scheme can be got as
\begin{equation}\label{equO2halfltimesemidis}
\frac{1}{\tau^\alpha}\sum_{k=0}^n d^{\alpha,\lambda}_{k}G(\rho,t_{n-k})=
    (\Delta+\gamma)^{\frac{\beta}{2}} G(\rho,t_{n})+f(\rho,t_{n}).
\end{equation}
Combining \eqref{equO2halfltimesemidis} with \eqref{defdis}, we obtain the fully discrete scheme of Eq. \eqref{equFKtodis}, i.e., SBD scheme
\begin{equation}\label{equO2fulldis}
    \frac{1}{\tau^\alpha}\sum_{k=0}^n {\rm diag}(\mathcal{V}d^{\alpha,\lambda}_{k})\mathbf{G_h^{n-k}}=(\Delta+\gamma)^{\frac{\beta}{2}}_h\mathbf{G_h^n}+\mathcal{V}f^n.
\end{equation}
\begin{theorem}\label{thmSBDsemER}
	Let $\mathbf{G_h}$ and $\mathbf{G^n_h}$ be the solutions of Eq. \eqref{equFkspatiallysemidis} and Eq. \eqref{equO2fulldis}, respectively. If $f\in L^{\infty}(0,T,L^2(\Omega))$ with $\int_{0}^t(t-s)^{\alpha-1}\|\mathcal{V}f''(s)\|_{l^2}ds\leq\infty$ for $t\in(0,T]$, then we get
	\begin{equation*}
	\|\mathbf{G_h}(t_n)-\mathbf{G^n_h}\|_{l^2}\leq C\tau^2\left(t_n^{\alpha-1}\|\mathcal{V}f'(0)\|_{l^2}+\int_0^{t_n}(t_n-s)^{\alpha-1}\|\mathcal{V}f''(s)\|_{l^2}ds\right).
	\end{equation*}
\end{theorem}

\begin{proof}
	According to \eqref{equspasemsollap}, we have
\begin{equation}\label{equO2exactsol}
\mathbf{G_h}(t_n)=\frac{1}{2\pi J}\int_{\Gamma_{\theta,\kappa}}e^{zt_n}\tilde{E}_h(z)z^{-2}dz\mathcal{V}f'(0)+\left(\frac{1}{2\pi J}\int_{\Gamma_{\theta,\kappa}}e^{zt}\tilde{E}_h(z)z^{-2}dz\ast \mathcal{V}f''(t)\right)(t_n),
\end{equation}
where the fact $f(t)=f(0)+t f'(0)+t\ast f''(t)$, the property of convolution and \eqref{equFkf00} are used.
Multiplying both sides of \eqref{equO2fulldis} by $\zeta^n$ and summing from $0$ to $\infty$ lead to
\begin{equation*}
\begin{split}
\sum_{n=0}^{\infty}\zeta^n\mathbf{G_h^n}=& \left(\omega_h\left(\frac{(1-\zeta)+(1-\zeta)^2/2}{\tau}\right)-(\Delta+\gamma)^{\frac{\beta}{2}}_h\right)^{-1}\sum_{n=0}^{\infty}\zeta^n\mathcal{V}f^n\\
=&\tilde{E}_h\left(\frac{(1-\zeta)+(1-\zeta)^2/2}{\tau}\right)\sum_{n=0}^{\infty}\zeta^n\mathcal{V}f^n.
\end{split}
\end{equation*}
Thus
\begin{equation*}
\mathbf{G_h^{n}}=\tau\sum_{j=0}^{n}E_{n-j}\mathcal{V}f^j,
\end{equation*}
where
\begin{equation}\label{equdefofejo2}
\tilde{E}_h\left(\frac{(1-\zeta)+(1-\zeta)^2/2}{\tau}\right)=\tau\sum_{j=0}^{\infty}E_j\zeta^j.
\end{equation}
For convenience, we define
\begin{equation*}
Q(t',v)=\tau \cdot {\rm diag}\left(\sum_{0\leq t_j\leq t'}E_j\mathcal{V}v(t'-t_j)\right)={\rm diag}\left((\hat{E}_\tau\ast \mathcal{V}v)(t')\right),
\end{equation*}
where $\hat{E}_\tau=\tau\sum\limits_{j=0}^{\infty}E_j\delta_{t_j}$.  Denoting $\hat{t}(\cdot,t)=t$, then
\begin{equation}\label{equO2numsol}
\begin{split}
\mathbf{G_h^{n}}=&(\hat{E}_\tau\ast\mathcal{V}f)(t_n)\\
=&{\rm diag}((\hat{E}_\tau\ast \mathcal{V}\hat{t} )(t_n))\mathcal{V}f'(0)+{\rm diag}(\hat{E}_\tau\ast \mathcal{V}\hat{t})\ast(\mathcal{V}f'' )(t_n)\\
=&Q(t_n,\hat{t})\mathcal{V}f'(0)+(Q(t',\hat{t})\ast\mathcal{V}f''(t'))(t_n),
\end{split}
\end{equation}
where the fact $f(t)=f(0)+t f'(0)+t\ast f''(t)$, the property of convolution and \eqref{equFkf00} are used. According to \eqref{equO2exactsol} and \eqref{equO2numsol}, we have
\begin{equation*}
\begin{split}
&\mathbf{G_h}-\mathbf{G_h^{n}}\\
=&\left(\frac{1}{2\pi J}\int_{\Gamma_{\theta,\kappa}}e^{zt_n}\tilde{E}_h(z)z^{-2}dz-Q(t_n,\hat{t})\right )\mathcal{V}f'(0)\\
&+\left(\left(\frac{1}{2\pi J}\int_{\Gamma_{\theta,\kappa}}e^{zt}\tilde{E}_h(z)z^{-2}dz-Q(t,\hat{t})\right)\ast \mathcal{V}f''(t)\right)(t_n). \\
\end{split}	
\end{equation*}
Now  we need to consider the error between $ \frac{1}{2\pi J}\int_{\Gamma_{\theta,\kappa}}e^{zt'}\tilde{E}_h(z)z^{-2}dz $ and $ Q(t',\hat{t}) $ for $ t'\in [t_{n-1},t_n) $, $ n\geq 1 $. As for $n=1$, we have
\begin{equation*}
\left \|Q(t',t)-\frac{1}{2\pi J}\int_{\Gamma_{\theta,\kappa}}e^{zt'}\tilde{E}_h(z)z^{-2}dz\right \|_{l_2\rightarrow\l_2}\leq Ct'^{(\alpha-1)}\tau^2.
\end{equation*}
Then for $n>1$, by Taylor's expansion, we have
\begin{equation}\label{equO2disker}
\begin{split}
&\frac{1}{2\pi J}\int_{\Gamma_{\theta,\kappa}}e^{zt'}\tilde{E}_h(z)z^{-2}dz\\
=&\frac{1}{2\pi J}\int_{\Gamma_{\theta,\kappa}}e^{zt_n}\tilde{E}_h(z)z^{-2}dz+\frac{1}{2\pi J}(t'-t_n)\int_{\Gamma_{\theta,\kappa}}e^{zt_n}\tilde{E}_h(z)z^{-1}dz\\
&+\frac{1}{2\pi J}\int_{\Gamma_{\theta,\kappa}}\int_{t_n}^{t'}(t'-s)e^{zs}z^2ds\tilde{E}_h(z)z^{-2}dz.
\end{split}
\end{equation}
By simple calculation, we get
\begin{equation}
\left \|\frac{1}{2\pi J}\int_{\Gamma_{\theta,\kappa}}\int_{t_n}^{t'}(t'-s)e^{zs}z^2ds\tilde{E}_h(z)z^{-2}dz\right \|_{l_2\rightarrow l_2}\leq C\tau^2t'^{(\alpha-1)}.
\end{equation}
As for $ Q(t',\hat{t}) $, we have
\begin{equation}\label{equO2disQ}
	Q(t',\hat{t})=\lim_{t'\rightarrow t_n} Q(t',\hat{t})+(t'-t_n)Q(t',1).
\end{equation}
First we consider the error between $ \int_{\Gamma_{\theta,\kappa}}e^{zt_n}\tilde{E}_h(z)z^{-1}dz $ and $ Q(t',1) $. It can be noted that, for $\xi_\tau=e^{-\tau(\kappa+1)}$,
\begin{equation*}
\tau E_n=\frac{1}{2\pi J}\int_{|\zeta|=\xi_\tau}\zeta^{-n-1}\tilde{E}_h\left(\frac{(1-\zeta)+(1-\zeta)^2/2}{\tau}\right)d\zeta.
\end{equation*}
 Hence
 \begin{equation*}
	Q(t',1)=\frac{\tau^{-1}}{2\pi J}\int_{|\zeta|=\xi_\tau}\zeta^{-n-1}\mu(\zeta)\tilde{E}_h\left(\frac{(1-\zeta)+(1-\zeta)^2/2}{\tau}\right)\left(\frac{(1-\zeta)+(1-\zeta)^2/2}{\tau}\right)^{-1}d\zeta,
 \end{equation*}
 where we define $ \mu(\zeta)=\frac{1}{2}(3-\zeta)\zeta $, and use the fact
 \begin{equation*}
 	\begin{split}
		\sum_{j=0}^{n-1}\zeta^{-j-1}=\zeta^{-1}\frac{1-\zeta^{-n}}{1-\zeta^{-1}}=\zeta^{-n-1}\frac{\mu(\zeta)}{(1-\zeta)+(1-\zeta)^2/2}-\frac{1}{1-\zeta}
 	\end{split}
 \end{equation*}
 and $ \tilde{E}_h\left((1-\zeta)+(1-\zeta)^2/2)/\tau\right)/(1-\zeta) $ is analytic for small $\zeta$.  Taking $\zeta=e^{-z\tau}$, and denoting $z_\tau=\frac{(1-e^{-z\tau})+(1-e^{-z\tau})^2/2}{\tau}$, we have
 \begin{equation*}
 Q(t',1)=\frac{1}{2\pi J}\int_{\Gamma^\tau}e^{zt_{n}}\mu(e^{-z\tau})\tilde{E}_h\left(z_\tau\right)z_\tau^{-1}dz,
 \end{equation*}
 where $\Gamma^\tau=\{z=\kappa+1+iy:y\in\mathbb{R}~{\rm and}~|y|\leq \pi/\tau\}$. Next we deform the contour $\Gamma^\tau$ to
 $\Gamma^\tau_{\theta,\kappa}=\{z\in \mathbb{C}:\kappa\leq |z|\leq\frac{\pi}{\tau\sin(\theta)},|\arg z|=\theta\}\bigcup\{z\in \mathbb{C}:|z|=\kappa,|\arg z|\leq\theta\}$, then $Q(t',1)$ can be rewritten as
 \begin{equation*}
 Q(t',1)=\frac{1}{2\pi J}\int_{\Gamma^\tau_{\theta,\kappa}}e^{zt_{n}}\mu(e^{-z\tau})\tilde{E}_h\left(z_\tau\right)z_\tau^{-1}dz.
 \end{equation*}
Thus we have
\begin{equation*}
\begin{split}
&\frac{1}{2\pi J}\int_{\Gamma_{\theta,\kappa}}e^{zt_{n}}\tilde{E}_h(z)z^{-1}dz-Q(t',1)\\
=&\frac{1}{2\pi J}\int_{\Gamma_{\theta,\kappa}\backslash\Gamma^\tau_{\theta,\kappa}}e^{zt_{n}}\tilde{E}_h(z)z^{-1}dz\\
&+\frac{1}{2\pi J}\int_{\Gamma^\tau_{\theta,\kappa}}e^{zt_{n}}\left (\tilde{E}_h(z)z^{-1}-\mu(e^{-z\tau})\tilde{E}_h\left(z_\tau\right)z_\tau^{-1}\right )dz\\
=&\uppercase\expandafter{\romannumeral1}+\uppercase\expandafter{\romannumeral2}.
\end{split}
\end{equation*}
For $\uppercase\expandafter{\romannumeral1}$, there exists the estimate
\begin{equation}
\|\uppercase\expandafter{\romannumeral1}\|_{l_2\rightarrow l_2}\leq C\int_{\Gamma_{\theta,\kappa}\backslash\Gamma^\tau_{\theta,\kappa}}e^{-c|z|t'}|z|^{-\alpha-1}d|z|\leq C\tau t'^{(\alpha-1)}.
\end{equation}
For $\uppercase\expandafter{\romannumeral2}$, we obtain, with the mean value theorem, Lemma \ref{lemDeri}, $z_\tau=\frac{(1-e^{-z\tau})+(1-e^{-z\tau})^2/2}{\tau}=z+O(\tau^2 z^3)$ and $ c|z|\leq |z_\tau|\leq C|z| $ when $z\in \Gamma^\tau_{\theta,\kappa}$,
\begin{equation}
\begin{split}
&\left\|\tilde{E}_h(z)z^{-1}-\mu(e^{-z\tau})\tilde{E}_h\left(z_\tau\right)z_\tau^{-1}\right\|_{l_2\rightarrow l_2}\\
\leq&\left\|\tilde{E}_h(z)z^{-1}-\tilde{E}_h\left(z_\tau\right)z_\tau^{-1}\right\|_{l_2\rightarrow l_2}+\left\|\tilde{E}_h\left(z_\tau\right)z_\tau^{-1}-\mu(e^{-z\tau})\tilde{E}_h\left(z_\tau\right)z_\tau^{-1}\right\|_{l_2\rightarrow l_2}\\
\leq& C|z|^{-\alpha-2}|\tau^2 z^3|+C\tau|z|^{-\alpha}\\
\leq& C\tau|z|^{-\alpha}.
\end{split}
\end{equation}
Consequently,
\begin{equation*}
\begin{split}
\|\uppercase\expandafter{\romannumeral2}\|_{l_2\rightarrow l_2}\leq& C\tau\left|\int_\kappa^{\frac{\pi}{\tau\sin(\theta)}}e^{-crt'}|r|^{-\alpha}dr\right|+C\tau\left|\int_{-\theta}^{\theta}e^{\kappa\cos{(\eta)} t_{n}}\kappa^{1-\alpha}d\eta\right|\\
\leq& Ct'^{\alpha-1}\tau+C\tau\kappa e^{\kappa t'}\leq Ct'^{(\alpha-1)}\tau,
\end{split}
\end{equation*}
where the fact $\kappa e^{\kappa t'}\leq \kappa T^{1-\alpha}e^{\kappa T}t'^{\alpha-1}$ is used.

Next, we consider the error between $ \int_{\Gamma_{\theta,\kappa}}e^{zt_n}\tilde{E}_h(z)z^{-2}dz $ and $ \lim\limits_{t'\rightarrow t_n}Q(t',\hat{t}) $. Using the fact that $ \sum\limits_{j=0}^{n}j\zeta^j=(\zeta-\zeta^{n+1})/(1-\zeta)^2-n\zeta^{n+1}/(1-\zeta) $ and $ \tilde{E}_h\left((1-\zeta)+(1-\zeta)^2/2)/\tau\right)/(1-\zeta) $ and $ \tilde{E}_h\left((1-\zeta)+(1-\zeta)^2/2)/\tau\right)/(1-\zeta)^2 $ are analytic for small $\zeta$, we have
\begin{equation*}
\begin{split}
&\lim_{t'\rightarrow t_n}Q(t',\hat{t})=\tau^2\sum_{j=0}^{n-1}E_j(n-j)\\
=&\frac{\tau^{-1}}{2\pi J}\int_{|\zeta|=\xi_\tau}\zeta^{-n-1}\mu_1(\zeta)\tilde{E}_h\left(\frac{(1-\zeta)+(1-\zeta)^2/2}{\tau}\right)\left(\frac{(1-\zeta)+(1-\zeta)^2/2}{\tau}\right)^{-2}d\zeta,
\end{split}
\end{equation*}
where $ \mu_1(\zeta)=\zeta(3-\zeta)^2/4 $. Taking $\zeta=e^{-z\tau}$, we have
\begin{equation*}
\lim_{t'\rightarrow t_n}Q(t',\hat{t})=\frac{1}{2\pi J}\int_{\Gamma^\tau}e^{zt_{n}}\mu_1(e^{-z\tau})\tilde{E}_h\left(z_\tau\right)z_\tau^{-2}dz,
\end{equation*}
where $\Gamma^\tau=\{z=\kappa+1+iy:y\in\mathbb{R}~{\rm and}~|y|\leq \pi/\tau\}$. Next we deform the contour $\Gamma^\tau$ to
$\Gamma^\tau_{\theta,\kappa}=\{z\in \mathbb{C}:\kappa\leq |z|\leq\frac{\pi}{\tau\sin(\theta)},|\arg z|=\theta\}\bigcup\{z\in \mathbb{C}:|z|=\kappa,|\arg z|\leq\theta\}$, then $\lim_{t'\rightarrow t_n}Q(t',\hat{t})$ can be rewritten as
\begin{equation*}
\lim_{t'\rightarrow t_n}Q(t',\hat{t})=\frac{1}{2\pi J}\int_{\Gamma^\tau_{\theta,\kappa}}e^{zt_{n}}\mu_1(e^{-z\tau})\tilde{E}_h\left(z_\tau\right)z_\tau^{-2}dz.
\end{equation*}
Thus we have
\begin{equation*}
\begin{split}
&\frac{1}{2\pi J}\int_{\Gamma_{\theta,\kappa}}e^{zt_{n}}\tilde{E}_h(z)z^{-2}dz-\lim_{t'\rightarrow t_n}Q(t',\hat{t})\\
=&\frac{1}{2\pi J}\int_{\Gamma_{\theta,\kappa}\backslash\Gamma^\tau_{\theta,\kappa}}e^{zt_{n}}\tilde{E}_h(z)z^{-2}dz\\
&+\frac{1}{2\pi J}\int_{\Gamma^\tau_{\theta,\kappa}}e^{zt_{n}}\left (\tilde{E}_h(z)z^{-2}-\mu_1(e^{-z\tau})\tilde{E}_h\left(z_\tau\right)z_\tau^{-2}\right )dz\\
=&\uppercase\expandafter{\romannumeral3}+\uppercase\expandafter{\romannumeral4}.
\end{split}
\end{equation*}
For $\uppercase\expandafter{\romannumeral3}$, there exists the estimate
\begin{equation}
\|\uppercase\expandafter{\romannumeral3}\|_{l_2\rightarrow l_2}\leq C\int_{\Gamma_{\theta,\kappa}\backslash\Gamma^\tau_{\theta,\kappa}}e^{-c|z|t'}|z|^{-\alpha-2}d|z|\leq C\tau^2 t'^{(\alpha-1)}.
\end{equation}
For $\uppercase\expandafter{\romannumeral4}$, we obtain, with the mean value theorem, Lemma \ref{lemDeri}, $z_\tau=\frac{(1-e^{-z\tau})+(1-e^{-z\tau})^2/2}{\tau}=z+O(\tau^2 z^3)$ and $ c|z|\leq |z_\tau|\leq C|z| $ for $z\in \Gamma^\tau_{\theta,\kappa}$,
\begin{equation}
\begin{split}
&\left\|\tilde{E}_h(z)z^{-2}-\mu_1(e^{-z\tau})\tilde{E}_h\left(z_\tau\right)z_\tau^{-2}\right\|_{l_2\rightarrow l_2}\\
\leq&\left\|\tilde{E}_h(z)z^{-2}-\tilde{E}_h\left(z_\tau\right)z_\tau^{-2}\right\|_{l_2\rightarrow l_2}+\left\|\tilde{E}_h\left(z_\tau\right)z_\tau^{-2}-\mu_1(e^{-z\tau})\tilde{E}_h\left(z_\tau\right)z_\tau^{-2}\right\|_{l_2\rightarrow l_2}\\
\leq& C|z|^{-\alpha-3}|\tau^2 z^3|+C\tau^2|z|^{-\alpha}\\
\leq& C\tau^2|z|^{-\alpha}.
\end{split}
\end{equation}
Consequently,
\begin{equation*}
\begin{split}
\|\uppercase\expandafter{\romannumeral4}\|_{l_2\rightarrow l_2}\leq& C\tau^2\left|\int_\kappa^{\frac{\pi}{\tau\sin(\theta)}}e^{-crt'}|r|^{-\alpha}dr\right|+C\tau^2\left|\int_{-\theta}^{\theta}e^{\kappa\cos(\eta) t_{n}}\kappa^{1-\alpha}d\eta\right|\\
\leq& Ct'^{(\alpha-1)}\tau^2+C\tau^2\kappa e^{\kappa t'}\leq Ct'^{(\alpha-1)}\tau^2,
\end{split}
\end{equation*}
where the fact $\kappa e^{\kappa t'}\leq \kappa T^{1-\alpha}e^{\kappa T}t'^{\alpha-1}$ is used. So we complete the proof.
\end{proof}

\begin{theorem}
  Let $G$ be the solution of Eq. \eqref{equFKtodis} satisfying $G(\cdot,x)\in C^2(\bar{\Omega})$  and $\mathbf{G^n_h}$ be the solution of Eq. \eqref{equO2fulldis}. If $f\in L^{\infty}(0,T,L^2(\Omega))$ satisfying
   \begin{equation*}
    \int_{0}^t(t-s)^{\alpha-1}\|f(s)\|_{L^2}ds\leq\infty
  \end{equation*}
  and
  \begin{equation*}
    \int_{0}^t(t-s)^{\alpha-1}\|\mathcal{V}f''(s)\|_{l^2}ds\leq\infty
  \end{equation*}
  for $t\in(0,T]$, then
  \begin{equation*}
  \begin{split}
    \|\mathcal{V}G(t_n)-\mathbf{G^n_h}\|_{l^2}\leq& C\tau^2 \left(t_n^{\alpha-1}\|\mathcal{V}f'(0)\|_{l^2}+\int_0^{t_n}(t_n-s)^{\alpha-1}\|\mathcal{V}f''(s)\|_{l^2}ds\right)\\&+Ch^{2-\beta}\int_0^{t_n}(t_n-s)^{\alpha-1}\|f(s)\|_{L^2}ds.
  \end{split}
  \end{equation*}
\end{theorem}
\begin{proof}
Combining Theorem \ref{thmspatialEr} and Theorem \ref{thmSBDsemER} leads to
\begin{equation*}
  \begin{aligned}
    \|\mathcal{V}G(t_n)-\mathbf{G^n_h}\|_{l^2}=&\|\mathcal{V}G(t_n)-\mathbf{G_h}(t_n)+\mathbf{G_h}(t_n)-\mathbf{G^n_h}\|_{l^2}\\
    \leq&\|\mathcal{V}G(t_n)-\mathbf{G_h}(t_n)\|_{l^2}+\|\mathbf{G_h}(t_n)-\mathbf{G^n_h}\|_{l^2}\\
    \leq& C\tau^2 \left(t_n^{\alpha-1}\|\mathcal{V}f'(0)\|_{l^2}+\int_0^{t_n}(t_n-s)^{\alpha-1}\|\mathcal{V}f''(s)\|_{l^2}ds\right)\\
    &+Ch^{2-\beta}\int_0^{t_n}(t_n-s)^{\alpha-1}\|f(s)\|_{L^2}ds.
  \end{aligned}
\end{equation*}
\end{proof}

\section{Efficient computations}

When discretizing the non-local operator, it generally gives rise to a full matrix, so an effective algorithm is needed to numerically solve \eqref{equFKtodis} satisfying homogeneous Dirichlet boundary conditions, especially for high dimensional cases. In this section, we state how to reduce the complexity of our algorithm.

We first give a lemma about the property of $w^{\beta,\gamma}_{p,q,i,j}$ in \eqref{equweightoffl}.
\begin{lemma}\label{lemweporps1}
  Assume $-N<p_1,i_1,p_2,i_2<N$, $-N<q_1,j_1,q_2,j_2<N$. If $(|p_1-i_1|,|q_1-j_1|)=(|p_2-i_2|,|q_2-j_2|)$ and $(|p_1-i_1|,|q_1-j_1|)\neq(0,0)$, then there is
  \begin{equation*}
w^{\beta,\gamma}_{p_1,q_1,i_1,j_1}=w^{\beta,\gamma}_{p_2,q_2,i_2,j_2}.
  \end{equation*}
\end{lemma}
\begin{proof}
  We first prove
  \begin{equation}\label{equrelaofW}
    W^1_{p,q}=W^2_{-p,q}=W^3_{p,-q}=W^4_{-p,-q}.
  \end{equation}
  According to \eqref{equdefWsp}, for $|p|\leq1$ and $|q|\leq1$, \eqref{equrelaofW} holds.
By \eqref{equdefW}, there exists
  \begin{equation*}
    \begin{aligned}
      &W^1_{p,q}-W^2_{-p,q}\\
      =&\left(H^{\xi\eta}_{p,q}-\xi_{p+1}H^\eta_{p,q}-\eta_{q+1}H^\xi_{p,q}+\xi_{p+1}\eta_{q+1}H_{p,q}\right)\\
      &+\left(H^{\xi\eta}_{-p-1,q}-\xi_{-p-1}H^\eta_{-p-1,q}-\eta_{q+1}H^\xi_{-p-1,q}+\xi_{-p-1}\eta_{q+1}H_{-p-1,q}\right).
    \end{aligned}
  \end{equation*}
From \eqref{equdefG}, we have
  \begin{equation*}
    \begin{aligned}
      &H^{\xi\eta}_{p,q}=-H^{\xi\eta}_{-p-1,q},~~H^\eta_{p,q}=H^\eta_{-p-1,q},\\
      &H^\xi_{p,q}=-H^\xi_{-p-1,q},~~H_{p,q}=H_{-p-1,q}.
    \end{aligned}
  \end{equation*}
Then there exists $W^1_{p,q}=W^2_{-p,q}$. Similarly, we have $W^1_{p,q}=W^3_{p,-q}=W^4_{-p,-q}$. Combining \eqref{equweightoffl} with \eqref{equrelaofW}, the lemma can be proved.
\end{proof}

When numerically solving Eq. \eqref{equFKtodis}, the full-discretization scheme \eqref{equO1fulldis} or \eqref{equO2fulldis} can be written as the matrix form
\begin{equation}\label{equlinsys}
\mathbf{A}\mathbf{G^n_h}=\mathbf{F},
\end{equation}
where
\begin{equation*}
  \mathbf{A}=\frac{1}{h^\beta}\mathbf{A_s}+\frac{1}{\tau^\alpha}\mathbf{A_t}.
\end{equation*}
Here, the elements of $\mathbf{A_s}$ correspond to the discretization of the tempered fractional Laplacian; the elements of $\mathbf{A_t}$ are with the discretization of the tempered fractional substantial derivative when $k=0$ for  Eq. \eqref{equO1fulldis} or Eq. \eqref{equO2fulldis}; the element of $\mathbf{F}$ is composed of discretizing the source term $f$ defined by \eqref{equFKtodis} and  the tempered fractional substantial derivative when $k\neq0$ for  Eq. \eqref{equO1fulldis} or Eq. \eqref{equO2fulldis}; and 
 \begin{equation*}
 \mathbf{A_t}={\rm diag}(\mathcal{V}d_0^{\alpha,\lambda}).
 \end{equation*}
 Next, we divide the matrix $\mathbf{A_s}$ into
\begin{equation*}
  \mathbf{A_s}=\mathbf{A_0}+\mathbf{A_d},
\end{equation*}
where
 \begin{equation*}
 \mathbf{A_0}=\left[\begin{matrix}
 0& w_{-N+1,-N+1,-N+1,-N+2} &\cdots & w_{-N+1,-N+1,N-1,N-1} \\
 w_{-N+1,-N+2,-N+1,-N+1}& 0 &\cdots & w_{-N+1,-N+2,N-1,N-1} \\
 \vdots& \vdots &\ddots & \vdots \\
 w_{N-1,N-1,-N+1,-N+1}&  w_{N-1,N-1,-N+1,-N+2} &\cdots &0 \\
 \end{matrix}\right],
 \end{equation*}

  \begin{equation*}
 \mathbf{A_d}=\left[\begin{matrix}
 w_{-N+1,-N+1,-N+1,-N+1}& 0 &\cdots & 0 \\
 0& w_{-N+1,-N+2,-N+1,-N+2} &\cdots & 0 \\
 \vdots&\vdots &\ddots & \vdots \\
 0& 0 &\cdots & w_{N-1,N-1,N-1,N-1} \\
 \end{matrix}\right].
 \end{equation*}

Based on Lemma \ref{lemweporps1} and the structure of matrix $\mathbf{A_0}$, it is easy to find that $\mathbf{A_0}$ is a symmetric block Toeplitz matrix with Toeplitz block (BTTB) matrix. Being similar to \cite{Chen2005}, the memory requirement for the  $(2N-1)\times (2N-1)$ matrix $\mathbf{A_0}$ can be reduced from $O(N^2)$ to $O(N)$ ($N$ is the dimension of matrix).


 When solving Eq. \eqref{equlinsys}, we use the Krylov subspace iterative methods to reduce computational costs, such as the conjugate gradient (CG) method and the PCG method. In the iteration process, the $\mathbf{A}\mathbf{v}$ needs to be calculated ($\mathbf{v}$ is a vector). By the above decomposition, one can calculate
\begin{equation*}
  \mathbf{A}\mathbf{v}=(\mathbf{A_0}\mathbf{v}+\mathbf{A_d}\mathbf{v})/h^\beta+(\mathbf{A_t}\mathbf{v})/\tau^\alpha.
\end{equation*}
Since $\mathbf{A_0}$ is a BTTB matrix, one can calculate $\mathbf{A}_0\mathbf{v}$ by FFT and the computation costs can be reduced from $O(N^2)$ to $O(N\ln N)+O(N)$. To reduce the total number of iteration steps, one needs to consider how to construct a suitable preconditioner. Ref. \cite{Chen2005} builds a preconditioner for a BTTB matrix. However, matrix $\mathbf{A}$ isn't a BTTB matrix (due to the entries on the main diagonal), so the preconditioner constructed in \cite{Chen2005} can not be directly used. 
Instead, we denote
 \begin{equation*}
 \tilde{\mathbf{A}}=\mathbf{A_0}+\frac{\sum\limits_{(p,q)=(-N+1,-N+1)}^{(N-1,N-1)}\left(w_{p,q,p,q}+d^{\alpha,\lambda}_{0}(x_{p},y_q)\right)}{4(N-1)^2}\mathbf{I},
 \end{equation*}
 where $\mathbf{I}$ is an identity matrix. It is easy to find that $\tilde{\mathbf{A}}$ is a BTTB matrix, so one can take a preconditioner of $\tilde{\mathbf{A}}$ as the one of $\mathbf{A}$.
 In numerical experiments, the effectiveness of the preconditioner is verified.

\section{Numerical experiments}

In this section, we verify the theoretical results on convergence rate and the effectiveness of the scheme by solving \eqref{equFKtodis} without the assumption on the regularity of the solution in time. Here, we consider the domain $\Omega=(-1,1)\times(-1,1)$ and time $T=1$; $l_2$ norm and $l_\infty$ norm are used to measure the numerical errors.
\subsection{Spatial convergence order}
\begin{example}\label{example:Hom}
Choose $U(\mathbf{x_0})=1$, $r(\mathbf{x_0})=-1$, and take the initial condition as
\begin{equation*}
G_0(\rho,\mathbf{x_0})=(1-x^2)(1-y^2) ~~~~\mathbf{x_0}\in\Omega;
\end{equation*}
 the source term is
\begin{equation*}
  \begin{aligned}
    f(t,\mathbf{x_0})=&\frac{\Gamma(1+\nu)}{\Gamma(1+\nu-\alpha)}e^{-(\lambda-r(\mathbf{x_0})-J\rho U(\mathbf{x_0}))t}t^{\nu-\alpha}(1-x^2)(1-y^2)\\
    &-(\Delta+\gamma)^{\frac{\beta}{2}}\left(e^{-(\lambda-r(\mathbf{x_0})-J\rho U(\mathbf{x_0}))t}(t^{\nu}+1)(1-x^2)(1-y^2)\right)\\
    &-\lambda^\alpha\left(e^{-(\lambda-r(\mathbf{x_0})-J\rho U(\mathbf{x_0}))t}(t^{\nu}+1)(1-x^2)(1-y^2)\right).
  \end{aligned}
\end{equation*}
Then Eq. (\ref{equFk}) has the exact solution
\begin{equation*}
G(\rho,t,\mathbf{x_0})=e^{-(\lambda-r(\mathbf{x_0})-J\rho U(\mathbf{x_0}))t}(t^{\nu}+1)(1-x^2)(1-y^2) ~~~~\mathbf{x_0}\in\Omega.
\end{equation*}
By \eqref{equinitial0equ}, there exists
\begin{equation*}
  G(\rho,t,\mathbf{x_0})=W(\rho,t,\mathbf{x_0})+e^{-(\lambda-r(\mathbf{x_0})-J\rho U(\mathbf{x_0}))t}(1-x^2)(1-y^2).
\end{equation*}
So $W(\rho,t,\mathbf{x_0})$ solves
\begin{equation*}
    \mathcal{L}^{\alpha,\lambda}_t W=(\Delta+\gamma)^{\frac{\beta}{2}}W-f_w(t,\mathbf{x_0}),
\end{equation*}
where
\begin{equation*}
  \begin{aligned}
    f_w(t,\mathbf{x_0})=&f-(\Delta+\gamma)^{\frac{\beta}{2}}\left(e^{-(\lambda-r(\mathbf{x_0})-J\rho U(\mathbf{x_0}))t}(1-x^2)(1-y^2)\right)\\
    &-\lambda^\alpha\left(e^{-(\lambda-r(\mathbf{x_0})-J\rho U(\mathbf{x_0}))t}(1-x^2)(1-y^2)\right).
  \end{aligned}
\end{equation*}

Here, in order to reduce the effect of time-discrete errors on spatial convergence rate, we use SBD method to discrete the $\mathcal{L}^{\alpha,\lambda}_t$, i.e., \eqref{equO2fulldis}. We choose $\nu=1.5$, $\alpha=0.3$ and $\lambda=0.1$ to make $f_w$ satisfy the conditions of Theorem \ref{thmSBDsemER}, which ensures the accuracy of the scheme. At the same time, we take  $\tau=1/640$ and $\sigma=1+\frac{\beta}{2}$. Table \ref{DHspatialconver} shows the spatial convergence rates of  solving Eq. \eqref{equFKtodis}; it can be noted that the results are consistent with the theoretical ones.

\begin{table}[htbp]\fontsize{8pt}{10pt}\selectfont
 \begin{center}
  \caption {Numerical errors and convergence rates with $\alpha=0.3$, $\lambda=0.1$, and $\sigma=1+\frac{\beta}{2}$} \vspace{5pt}
\begin{tabular}{ccccccc}
\toprule[1pt]
        h   &            &        1/8 &       1/16 &       1/32 &       1/64 &      1/128 \\
\hline
           &$l^\infty$            &  3.464E-03 &  1.380E-03 &  5.262E-04 &  1.956E-04 &  7.161E-05 \\

  $\beta=0.5$ &Rate            &            &    1.3281  &    1.3908  &    1.4276  &    1.4497  \\
\cline{2-7}
 $\gamma=0.05$ &$l^2$            &  4.470E-03 &  1.809E-03 &  6.961E-04 &  2.601E-04 &  9.546E-05 \\

           &Rate            &            &    1.3048  &    1.3780  &    1.4205  &    1.4458  \\
\hline
           &$l^\infty$            &  8.354E-03 &  3.960E-03 &  1.805E-03 &  8.061E-04 &  3.560E-04 \\

  $\beta=0.8$ &Rate            &            &    1.0770  &    1.1333  &    1.1630  &    1.1791  \\
\cline{2-7}
  $\gamma=0.05$ &$l^2$            &  1.037E-02 &  4.999E-03 &  2.301E-03 &  1.033E-03 &  4.576E-04 \\

           &Rate            &            &    1.0528  &    1.1194  &    1.1552  &    1.1748  \\
\hline
           &$l^\infty$            &  2.274E-02 &  1.400E-02 &  8.343E-03 &  4.887E-03 &  2.837E-03 \\

  $\beta=1.2$ &Rate            &            &    0.6993  &    0.7471  &    0.7717  &    0.7847  \\
\cline{2-7}
  $\gamma=0.05$ &$l^2$            &  2.670E-02 &  1.666E-02 &  1.001E-02 &  5.894E-03 &  3.432E-03 \\

           &Rate            &            &    0.6801  &    0.7349  &    0.7643  &    0.7802  \\
\hline
           &$l^\infty$            &  4.310E-02 &  3.256E-02 &  2.399E-02 &  1.742E-02 &  1.254E-02 \\

  $\beta=1.5$&Rate            &            &    0.4045  &    0.4412  &    0.4618  &    0.4743  \\
\cline{2-7}
  $\gamma=0.05$ &$l^2$            &  4.864E-02 &  3.705E-02 &  2.744E-02 &  2.000E-02 &  1.444E-02 \\

           &Rate            &            &    0.3927  &    0.4331  &    0.4562  &    0.4705  \\

\bottomrule[1pt]
\end{tabular}\label{DHspatialconver}

  \end{center}
\end{table}

 Table \ref{DHspatialtime} shows the CPU time(s) and average iteration times of solving Eq. \eqref{equFKtodis} when using CG method and PCG method. When the mesh size $h$ is small, PCG method has a significant advantage of time and average iteration times compared to CG method, which shows that our preconditioner is effective.

\begin{table}[htbp]\fontsize{8pt}{10pt}\selectfont
 \begin{center}
  \caption {Performance of the CG and PCG method} \vspace{5pt}
\begin{tabular}{ccccccc}
\toprule[1pt]

      h     &            &        1/8 &       1/16 &       1/32 &       1/64 &      1/128 \\
\hline
           &PCG time(s)     &     10.31  &     28.91  &     90.86  &    490.63  &   1641.22  \\

      $\beta=0.5$ &PCG iterations    &      7.00  &      8.00  &      8.00  &      9.00  &      9.00  \\
\cline{2-7}
     $\gamma=0.05$ &CG time(s)            &      9.86  &     29.47  &     99.77  &    538.13  &   2076.00  \\

           &CG iterations            &      9.00  &     11.00  &     12.00  &     14.00  &     17.00  \\
\hline
           &PCG time(s)            &     10.31  &     32.58  &    103.64  &    583.30  &   2047.86  \\

       $\beta=0.8$ &PCG iterations            &      9.00  &     10.00  &     11.00  &     12.00  &     14.00  \\
\cline{2-7}
      $\gamma=0.05$ &CG time(s)            &     11.64  &     36.98  &    133.17  &    878.98  &   3478.66  \\

           &CG iterations            &     12.00  &     17.00  &     22.00  &     29.00  &     38.00  \\
\hline
           &PCG time(s)            &     11.42  &     38.05  &    124.86  &    805.63  &   2676.86  \\

       $\beta=1.2$ &PCG iterations            &     12.00  &     14.00  &     17.00  &     20.00  &     22.00  \\
\cline{2-7}
      $\gamma=0.05$ &CG time(s)            &     13.22  &     44.03  &    208.13  &   1739.61  &   8164.75  \\

           &CG iterations            &     18.00  &     28.00  &     44.00  &     69.00  &    107.98  \\
\hline
           &PCG time(s)            &     12.48  &     44.69  &    144.45  &   1070.56  &   4163.16  \\

       $\beta=1.5$ &PCG iterations            &     14.00  &     19.00  &     22.00  &     30.00  &     41.00  \\
\cline{2-7}
      $\gamma=0.05$ &CG time(s)            &     14.92  &     56.19  &    295.34  &   3012.66  &  15843.39  \\

           &CG iterations            &     23.00  &     40.00  &     71.03  &    128.62  &    222.62  \\

\bottomrule[1pt]
\end{tabular}\label{DHspatialtime}

  \end{center}
\end{table}

\end{example}
\subsection{Time convergence order}
\begin{example}
Choose the exact solution given in Example \ref{example:Hom} to verify the time convergence orders by BE and SBD methods. Here, in order to reduce the effect of spatial-discrete errors on time convergence rates, we choose $h=1/256$.

 Firstly, we verify convergence orders of the BE scheme \eqref{equO1fulldis}. We take $\nu=0.8$ to satisfy the conditions needed in Theorem \ref{thmBEsemER}, and then let  $\beta=0.5$, $\gamma=0$ and $\sigma=1.25$. The results are shown in Table \ref{DHtimeO1}, which are consistent with our theoretical results. Afterwards, Table \ref{DHtimeO2} gives the numerical errors and convergence rates of the SBD scheme \eqref{equO2fulldis} when $\nu=1.8$, $\beta=0.2$, $\gamma=0$, and $\sigma=2$.

\begin{table}[htbp]\fontsize{8pt}{10pt}\selectfont\linespread{5}
 \begin{center}
  \caption {Numerical errors and convergence rates with $\beta=0.5$, $\gamma=0$, and $\sigma=1.25$} \vspace{5pt}
\begin{tabular}{cccccc}
\toprule[1pt]
 $\tau$          & &        1/5 &       1/10 &       1/20 &       1/40 \\
\hline

 &$l^\infty$&  7.345E-03 &  3.681E-03 &  1.846E-03 &  9.287E-04 \\

$\alpha=0.3$ &Rate&            &    0.9967  &    0.9958  &    0.9910  \\
\cline{2-6}
$\lambda=0.5$&$l^2$&  8.104E-03 &  4.064E-03 &  2.040E-03 &  1.028E-03 \\

 &Rate &           &    0.9958  &    0.9945  &    0.9885  \\
\hline
&$l^\infty$&  1.284E-02 &  6.427E-03 &  3.204E-03 &  1.597E-03 \\

$\alpha=0.5$&Rate &            &    0.9986  &    1.0041  &    1.0048  \\
\cline{2-6}
$\lambda=0.5$&$l^2$&  1.414E-02 &  7.081E-03 &  3.532E-03 &  1.762E-03 \\

 &Rate&            &    0.9976  &    1.0034  &    1.0037  \\
\hline
 &$l^\infty$&  1.998E-02 &  1.015E-02 &  5.100E-03 &  2.548E-03 \\

 $\alpha=0.7$&Rate&            &    0.9770  &    0.9928  &    1.0008  \\
\cline{2-6}
 $\lambda=0.5$&$l^2$&  2.198E-02 &  1.118E-02 &  5.623E-03 &  2.812E-03 \\

 &Rate&            &    0.9749  &    0.9916  &    0.9999  \\
\bottomrule[1pt]
\end{tabular}\label{DHtimeO1}

  \end{center}
\end{table}


\begin{table}[htbp]\fontsize{8pt}{10pt}\selectfont
 \begin{center}
  \caption {Numerical errors and convergence rates with $\beta=0.2$, $\gamma=0$, and $\sigma=2$} \vspace{5pt}
\begin{tabular}{cccccc}
\toprule[1pt]
    $\tau$       &            &        1/5 &       1/10 &       1/20 &       1/40 \\
\hline
           &$l^\infty$&  1.702E-03 &  4.329E-04 &  1.071E-04 &  2.697E-05 \\

  $\alpha=0.3$ &Rate            &            &    1.9747  &    2.0146  &    1.9898  \\
\cline{2-6}
  $\lambda=0.8$ &$l^2$&  1.845E-03 &  4.697E-04 &  1.164E-04 &  2.944E-05 \\

           &Rate            &            &    1.9738  &    2.0129  &    1.9829  \\
\hline
           &$l^\infty$&  2.872E-03 &  7.322E-04 &  1.804E-04 &  4.487E-05 \\

  $\alpha=0.5$ &Rate            &            &    1.9718  &    2.0208  &    2.0075  \\
\cline{2-6}
  $\lambda=0.8$ &$l^2$&  3.106E-03 &  7.920E-04 &  1.953E-04 &  4.867E-05 \\

           &Rate            &            &    1.9713  &    2.0200  &    2.0043  \\
\hline
           &$l^\infty$&  4.208E-03 &  1.073E-03 &  2.633E-04 &  6.508E-05 \\

  $\alpha=0.8$ &Rate            &            &    1.9720  &    2.0263  &    2.0165  \\
\cline{2-6}
  $\lambda=0.8$ &$l^2$&  4.542E-03 &  1.158E-03 &  2.843E-04 &  7.037E-05 \\

           &Rate            &            &    1.9717  &    2.0260  &    2.0147  \\

\bottomrule[1pt]
\end{tabular}\label{DHtimeO2}

  \end{center}
\end{table}

\end{example}

Following that, we verify the time and spatial convergence orders by the unknown exact solution.
\begin{example}
 Consider $U(\mathbf{x_0})=(x^2+y^2)$ and  $r(\mathbf{x_0})=-(x^2+y^2)$. Take the initial condition 
\begin{equation*}
G_0(\rho,\mathbf{x_0})=0 ~~~~\mathbf{x_0}\in\Omega;
\end{equation*}
 the source term is
\begin{equation*}
  \begin{aligned}
    f(t,\mathbf{x_0})=t^{\nu}.
  \end{aligned}
\end{equation*}
Since the exact solution is unknown, we use
\begin{equation*}
  e_h=\|G_{2h}-G_{h}\|
\end{equation*}
to measure the errors, where $G_h$ is the numerical solution under mesh size $h$.

Firstly, to verify the spatial convergence orders, we take $\nu=1.2$, $\tau=1/640$, $\alpha=0.5$, $\lambda=0.3$, and $\sigma=1+\beta/2$. The results are shown in Table \ref{unknown1spatial}. Since the regularity of the unknown solution does not meet the assumption of theoretical results, the convergence rates are lower.

\begin{table}[h]\fontsize{8pt}{10pt}\selectfont
 \begin{center}
  \caption {Numerical errors and convergence rates with $\alpha=0.5$, $\lambda=0.3$, and $\sigma=1+\beta/2$} \vspace{5pt}
\begin{tabular}{ccccccc}
\toprule[1pt]

    h       &            &          1/8 &         1/16 &         1/32 &         1/64 &        1/128 \\
\hline
           &$l^\infty$            &  3.844E-02 &  3.694E-02 &  3.383E-02 &  2.981E-02 &  2.709E-02 \\

  $\beta=0.5$ &Rate            &            &    0.0574  &    0.1268  &    0.1829  &    0.1379  \\
\cline{2-7}
  $\gamma=0.5$ &$l^2$            &  2.183E-02 &  1.689E-02 &  1.223E-02 &  8.459E-03 &  5.657E-03 \\

           &Rate            &            &    0.3700  &    0.4662  &    0.5314  &    0.5803  \\
\hline
           &$l^\infty$            &  4.165E-02 &  4.162E-02 &  3.727E-02 &  3.229E-02 &  2.756E-02 \\

  $\beta=0.8$ &Rate            &            &    0.0008  &    0.1595  &    0.2068  &    0.2286  \\
\cline{2-7}
  $\gamma=0.5$ &$l^2$            &  2.259E-02 &  1.860E-02 &  1.389E-02 &  9.674E-03 &  6.393E-03 \\

           &Rate            &            &    0.2803  &    0.4212  &    0.5219  &    0.5975  \\
\hline
           &$l^\infty$            &  4.003E-02 &  3.113E-02 &  2.515E-02 &  1.876E-02 &  1.327E-02 \\

  $\beta=1.2$&Rate            &            &    0.3629  &    0.3078  &    0.4230  &    0.4987  \\
\cline{2-7}
  $\gamma=0.5$ &$l^2$            &  2.490E-02 &  1.894E-02 &  1.332E-02 &  8.918E-03 &  5.761E-03 \\

           &Rate            &            &    0.3945  &    0.5080  &    0.5789  &    0.6305  \\
\hline
           &$l^\infty$            &  1.772E-02 &  1.504E-02 &  1.211E-02 &  9.288E-03 &  6.918E-03 \\

  $\beta=1.5$&Rate            &            &    0.2366  &    0.3131  &    0.3822  &    0.4249  \\
\cline{2-7}
  $\gamma=0.5$ &$l^2$            &  1.480E-02 &  1.331E-02 &  1.089E-02 &  8.416E-03 &  6.291E-03 \\

           &Rate            &            &    0.1527  &    0.2901  &    0.3713  &    0.4199  \\

\bottomrule[1pt]
\end{tabular}\label{unknown1spatial}

  \end{center}
\end{table}

Next, we verify the time convergence orders, i.e., the BE scheme \eqref{equO1fulldis} and SBD scheme \eqref{equO2fulldis}. Here, we take $\nu=0.2$ and $\nu=1.2$ to satisfy the conditions needed for Theorem \ref{thmBEsemER} and Theorem \ref{thmSBDsemER}, respectively, and then we let $\beta=0.5$, $\gamma=0.05$, and $\sigma=2$. The results are shown in Tables \ref{unknown2timeO1} and \ref{unknown1timeO2}, respectively, which are consistent with our theoretical results.

\begin{table}[h]\fontsize{8pt}{10pt}\selectfont
 \begin{center}
  \caption {Numerical errors and convergence rates with $\beta=0.5$, $\gamma=0.05$, and $\sigma=2$} \vspace{5pt}
\begin{tabular}{ccccccc}
\toprule[1pt]

  $\tau$         &            &         1/10 &         1/20 &         1/40 &         1/80 &        1/160 \\
\hline
           &$l^\infty$            & 4.7286E-03 & 2.1208E-03 & 9.7136E-04 & 4.5171E-04 & 2.1251E-04 \\

  $\alpha=0.1$ &Rate            &            &    1.1568  &    1.1265  &    1.1046  &    1.0878  \\
\cline{2-7}
  $\lambda=0.1$ &$l^2$            & 7.7837E-03 & 3.4997E-03 & 1.5907E-03 & 7.3012E-04 & 3.3830E-04 \\

           &Rate            &            &    1.1532  &    1.1376  &    1.1235  &    1.1098  \\
\hline
           &$l^\infty$            & 6.9103E-03 & 3.2737E-03 & 1.5560E-03 & 7.4300E-04 & 3.5654E-04 \\

  $\alpha=0.5$ &Rate            &            &    1.0778  &    1.0730  &    1.0664  &    1.0593  \\
\cline{2-7}
  $\lambda=0.1$ &$l^2$            & 1.1407E-02 & 5.2453E-03 & 2.4263E-03 & 1.1352E-03 & 5.3846E-04 \\

           &Rate            &            &    1.1209  &    1.1122  &    1.0958  &    1.0761  \\
\hline
           &$l^\infty$            & 7.8552E-03 & 3.9342E-03 & 1.9366E-03 & 9.4549E-04 & 4.6062E-04 \\

  $\alpha=0.9$ &Rate            &            &    0.9976  &    1.0226  &    1.0344  &    1.0375  \\
\cline{2-7}
  $\lambda=0.1$ &$l^2$            & 1.2449E-02 & 6.0706E-03 & 2.9746E-03 & 1.4761E-03 & 7.4225E-04 \\

           &Rate            &            &    1.0361  &    1.0291  &    1.0109  &    0.9918  \\

\bottomrule[1pt]
\end{tabular}\label{unknown2timeO1}

  \end{center}
\end{table}


\begin{table}[h]\fontsize{8pt}{10pt}\selectfont
 \begin{center}
  \caption {Numerical errors and convergence rates with $\beta=0.5$, $\gamma=0.05$, and $\sigma=2$} \vspace{5pt}
\begin{tabular}{ccccccc}
\toprule[1pt]

    $\tau$       &            &       1/10 &       1/20 &       1/40 &       1/80 &      1/160 \\
\hline
           &$l^\infty$            &  5.710E-04 &  1.533E-04 &  3.826E-05 &  9.483E-06 &  2.351E-06 \\

  $\alpha=0.1$ &Rate            &            &    1.8966  &    2.0029  &    2.0123  &    2.0124  \\
\cline{2-7}
  $\lambda=0.1$ &$l^2$            &  8.954E-04 &  2.336E-04 &  5.822E-05 &  1.446E-05 &  3.593E-06 \\

           &Rate            &            &    1.9386  &    2.0044  &    2.0098  &    2.0086  \\
\hline
           &$l^\infty$            &  1.901E-03 &  4.165E-04 &  9.991E-05 &  2.486E-05 &  6.276E-06 \\

  $\alpha=0.5$ &Rate            &            &    2.1905  &    2.0594  &    2.0066  &    1.9862  \\
\cline{2-7}
  $\lambda=0.1$ &$l^2$            &  2.105E-03 &  5.095E-04 &  1.254E-04 &  3.119E-05 &  7.804E-06 \\

           &Rate            &            &    2.0470  &    2.0227  &    2.0069  &    1.9991  \\
\hline
           &$l^\infty$            &  4.117E-03 &  9.535E-04 &  2.298E-04 &  5.692E-05 &  1.427E-05 \\

  $\alpha=0.9$ &Rate            &            &    2.1102  &    2.0525  &    2.0137  &    1.9961  \\
\cline{2-7}
  $\lambda=0.1$ &$l^2$            &  4.135E-03 &  9.653E-04 &  2.333E-04 &  5.768E-05 &  1.442E-05 \\

           &Rate            &            &    2.0989  &    2.0489  &    2.0158  &    2.0004  \\

\bottomrule[1pt]
\end{tabular}\label{unknown1timeO2}

  \end{center}
\end{table}

\end{example}

\section{Conclusion}
The model describing the functional distribution of the trajectory of the reaction and diffusion process was recently built \cite{Hou2018}, which is composed of tempered fractional substantial derivative in time and tempered fractional Laplacian in space. To develop the finite difference schemes for the two dimensional model, we use the convolution quadrature to approximate the tempered fractional substantial derivative and, respectively, get the first-order and second-order approximation, and the weighted trapezoidal rule and bilinear interpolation are used to deal with the tempered fractional Laplacian, which is based on our previous work and modifies the regularity requirement of the solution according to \cite{Zhang2017}.  The error analyses of the designed schemes are strictly performed.  Moreover, some techniques are introduced to effectively reduce the complexity of the algorithm. Finally, we verify the predicted convergence rates and the effectiveness of the proposed schemes by numerical experiments.


\section*{Acknowledgements}
This work was supported by the National Natural Science Foundation of China under grant no. 11671182, and the Fundamental Research Funds for the Central Universities under grants no. lzujbky-2018-ot03 and no. lzujbky-2017-ot10.

\section*{Appendix}
\appendix
  \section{Proof of Theorem \ref{thmtrunerror}}
  To prove Theorem \ref{thmtrunerror}, we need the following Lemmas. Lemma \ref{lemmaintepola} gives the error estimate of the bilinear interpolation.
  \begin{lemma}[\cite{Brenner2008}]\label{lemmaintepola}
Let $I$ denote the bilinear interpolation on the box $K=[0,h]\times[0,h]$. For $f\in W^{2,\infty}(K)$ ($W^{k,p}(K)$ denotes the Sobolev space), the error of bilinear interpolation is bounded by
\begin{equation}
  \|f-I f\|_{l_\infty}\leq ch^2\left(\left\|\frac{\partial^2 f}{\partial x^2}\right\|_{l_\infty}+\left\|\frac{\partial^2 f}{\partial y^2}\right\|_{l_\infty}\right).
\end{equation}
\end{lemma}

 Next, we provide the estimates about $\phi_{\sigma}$, $\frac{\partial^2}{\partial\xi^2}\frac{G(\xi,\eta)-G(x_p,y_q)}{e^{\gamma \sqrt{(\xi-x_p)^2+(\eta-x_q)^2}}}$, and $\frac{\partial^2}{\partial\eta^2}\frac{G(\xi,\eta)-G(x_p,y_q)}{e^{\gamma \sqrt{(\xi-x_p)^2+(\eta-x_q)^2}}}$.
\begin{lemma}\label{lemfunc2error}
Let $\beta\in(0,2)$ , $\xi> 0$, and $\eta>0$. If $G(x,y)\in C^{2}(\bar{\Omega})$, then for $(x,y)\in\bar{\Omega}$, there are
\begin{equation}\label{eqphikz}
\begin{split}
&\left|\phi_{\sigma}\right|\leq C\left(\xi^2+\eta^2\right)^{1-\frac{\sigma}{2}},\\
&\left|\frac{\partial^2}{\partial\xi^2}\frac{G(\xi,\eta)-G(x_p,y_q)}{e^{\gamma \sqrt{(\xi-x_p)^2+(\eta-x_q)^2}}}\right|\leq C,
\\
&\left|\frac{\partial^2}{\partial\eta^2}\frac{G(\xi,\eta)-G(x_p,y_q)}{e^{\gamma \sqrt{(\xi-x_p)^2+(\eta-x_q)^2}}}\right|\leq C,
\end{split}
\end{equation}
with $C$ being a positive constant.
\end{lemma}

\begin{proof}
  The proof of the first inequality of \eqref{eqphikz} can be found in \cite{Sun2018}. For the second one, by simple calculation, we obtain
  \begin{equation*}
    \begin{aligned}
      &\frac{\partial^2}{\partial\xi^2}\frac{G(\xi,\eta)-G(x_p,y_q)}{e^{\gamma \sqrt{(\xi-x_p)^2+(\eta-x_q)^2}}}\\
      =&e^{-\gamma \sqrt{(\xi-x_p)^2+(\eta-x_q)^2}}\frac{\partial^2 G(\xi,\eta)}{\partial\xi^2}\\
      &-2\frac{\gamma (\xi-x_p) e^{-\gamma  \sqrt{(\xi-x_p)^2+(\eta-y_q)^2}}}{\sqrt{(\xi-x_p)^2+(\eta-y_q)^2}}\frac{\partial G(\xi,\eta)}{\partial\xi}\\
      &+\left(\frac{\gamma^2  (\xi-x_p)^2 e^{-\gamma \sqrt{(\xi-x_p)^2+(\eta-y)^2}}}{(\xi-x_p)^2+(\eta-y_q)^2}+\frac{\gamma  (\xi-x_p)^2 e^{-\gamma  \sqrt{(\xi-x_p)^2+(\eta-y_q)^2}}}{\left((\xi-x_p)^2+(\eta-y_q)^2\right)^{3/2}}\right.\\&\left.-\frac{\gamma  e^{-\gamma \sqrt{(\xi-x_p)^2+(\eta-y_q)^2}}}{\sqrt{(\xi-x_p)^2+(\eta-y_q)^2}}\right)\left(G(\xi,\eta)-G(x_p,y_q)\right).
    \end{aligned}
  \end{equation*}
 According to mean value theorem, it yields that
  \begin{equation*}
    \begin{aligned}
      &G(\xi,\eta)-G(x_p,y_q)\\
      =&G(\xi,\eta)-G(\xi,y_q)+G(\xi,y_q)-G(x_p,y_q)\\
      =&(\eta-y_q)\frac{\partial}{\partial\eta}G(\xi,\bar{\eta})+(\xi-x_p)\frac{\partial}{\partial\xi}G(\bar{\xi},y_q),
    \end{aligned}
  \end{equation*}
  where $\bar{\xi}\in(\xi,x_p)$ and $\bar{\eta}\in(\eta,y_q)$. Since $G(x,y)\in C^{2}(\bar{\Omega})$, we have
  \begin{equation*}
    \left|\frac{\partial^2}{\partial\xi^2}\frac{G(\xi,\eta)-G(x_p,y_q)}{e^{\gamma  \sqrt{(\xi-x_p)^2+(\eta-x_q)^2}}}\right|\leq C.
  \end{equation*}
The proof of the third inequality of (\ref{eqphikz}) is similar to the second one.
\end{proof}

Now, we begin to prove Theorem \ref{thmtrunerror}.
\begin{proof}
 From \eqref{spatemdef2}, \eqref{spatemdef3}, \eqref{equtodis} and \eqref{equdiswithI}, for any $p,~q$, we obtain the error function
\begin{equation}\label{error1}
\begin{split}
  e^h_{\beta,\gamma}(x_p,y_q)=&(\Delta+\gamma)^{\frac{\beta}{2}}G(x_p,y_q)-(\Delta+\gamma)_{h}^{\frac{\beta}{2}}G(x_p,y_q)\\
        =&\left(\int_{0}^{h}\int_{0}^{h}\phi_{\sigma}(\xi,\eta)(\xi^2+\eta^2)^{\frac{\sigma-2-\beta}{2}}d\eta d\xi\right.\\
        &\left.-\int_{0}^{h}\int_{0}^{h}\frac{k_\sigma}{4}\left(\phi_{\sigma}(\xi_0,\eta_1)+\phi_{\sigma}(\xi_1,\eta_0)+\phi_{\sigma}(\xi_1,\eta_1)\right)(\xi^2+\eta^2)^{\frac{\sigma-2-\beta}{2}}d\eta d\xi\right) \\
        &+\sum_{
        \begin{subarray}
         ~i=-N;j=-N;\\(i,j)\notin\mathcal{I}_{p,q}
        \end{subarray}
        }^{i=N-1;j=N-1}\left(\int_{\xi_{i}}^{\xi_{i+1}}\int_{\eta_{j}}^{\eta_{j+1}}\frac{G(\xi,\eta)-G(x_p,y_q)}{\vartheta(x_p,y_q,\xi,\eta)}d\eta d\xi-I_{p,q,i,j}\right) \\
        =&\uppercase\expandafter{\romannumeral1}+\uppercase\expandafter{\romannumeral2}.
\end{split}
\end{equation}
For the first part of \eqref{error1}, there exists
\begin{equation*}
\begin{split}
|\uppercase\expandafter{\romannumeral1}|&\leq\int_{0}^{h}\int_{0}^{h}\left(\left|\phi_{\sigma}(\xi,\eta)\right|+\frac{k_{\sigma}}{4}\left|\phi_{\sigma}(\xi_0,\eta_1)+\phi_{\sigma}(\xi_1,\eta_0)+\phi_{\sigma}(\xi_1,\eta_1)\right|\right)(\xi^2+\eta^2)^{\frac{\sigma-2-\beta}{2}} d\eta d\xi \\
&\leq\int_{0}^{h}\int_{0}^{h}\left(C(\xi^2+\eta^2)^{1-\frac{\sigma}{2}}+Ch^{2-\sigma}\right)(\xi^2+\eta^2)^{\frac{\sigma-2-\beta}{2}} d\eta d\xi .
\end{split}
\end{equation*}
Taking $\xi =\bar{p}h$, $\eta =\bar{q}h$, we have
\begin{equation*}
\begin{split}
|\uppercase\expandafter{\romannumeral1}|\leq& Ch^{2-\beta}\int_{{0}}^{1}\int_{{0}}^{1}(\bar{p}^2+\bar{q}^2)^{-\frac{\beta}{2}}d\bar{q} d\bar{p} \\
                                          & +Ch^{2-\beta}\int_{{0}}^{1}\int_{{0}}^{1}(\bar{p}^2+\bar{q}^2)^{\frac{\sigma-2-\beta}{2}}d\bar{q} d\bar{p} .
\end{split}
\end{equation*}
Since $\beta<\sigma\leq2$, we obtain $-\beta>-2$ and $\sigma-2-\beta>-2$. Then, it holds
\begin{equation*}
|\uppercase\expandafter{\romannumeral1}|\leq Ch^{2-\beta}.
\end{equation*}
For the second part of \eqref{error1}, according to Lemma \ref{lemmaintepola}, we have
\begin{equation*}
\begin{split}
|\uppercase\expandafter{\romannumeral2}|\leq C&\sum_{
        \begin{subarray}
         ~i=-N;j=-N;\\(i,j)\notin \mathcal{I}_{p,q}
        \end{subarray}
        }^{i=N-1;j=N-1}\int_{\xi_{i}}^{\xi_{i+1}}\int_{\eta_{j}}^{\eta_{j+1}}\left(\left\|\frac{\partial^2}{\partial\xi^2}\frac{G(\xi,\eta)-G(x_p,y_q)}{e^{\gamma \sqrt{(\xi-x_p)^2+(\eta-x_q)^2}}}\right\|_{L_\infty}\right.\\
        &\left.+\left\|\frac{\partial^2}{\partial\eta^2}\frac{G(\xi,\eta)-G(x_p,y_q)}{e^{\gamma  \sqrt{(\xi-x_p)^2+(\eta-x_q)^2}}}\right\|_{L_\infty}\right)h^2((\xi-x_p)^2+(\eta-x_q)^2)^{\frac{-2-\beta}{2}}d\eta d\xi .
\end{split}
\end{equation*}
Further using Lemma \ref{lemfunc2error} leads to
\begin{equation*}
\begin{split}
|\uppercase\expandafter{\romannumeral2}|\leq C&\sum_{
        \begin{subarray}
         ~i=-N;j=-N;\\(i,j)\notin \mathcal{I}_{p,q}
        \end{subarray}
        }^{i=N-1;j=N-1}\int_{\xi_{i}}^{\xi_{i+1}}\int_{\eta_{j}}^{\eta_{j+1}}h^2((\xi-x_p)^2+(\eta-x_q)^2)^{\frac{-2-\beta}{2}}d\eta d\xi\\
        &\leq Ch^{2-\beta}.
\end{split}
\end{equation*}
So for $G(x,y)\in C^2(\bar{\Omega})$, it holds that
\begin{equation*}
\begin{aligned}
\left\|e^h_{\beta,\gamma}\right\|_{l_\infty}\leq Ch^{2-\beta}, ~
\left\|e^h_{\beta,\gamma}\right\|_{l_2}\leq Ch^{2-\beta}.
\end{aligned}
\end{equation*}
\end{proof}

  \section {Positive definiteness of matrix $\mathbf{A_s}$}\label{appendApos}
  \begin{lemma}[\cite{Axelsson1996}]\label{lemmaGersgorin}
  The spectrum $\lambda(A)$ of the matrix $A=[a_{i,j}]$ is enclosed in the union of the discs
  \begin{equation*}
  C_i=\left\{z\in \mathbb{C};\,|z-a_{i,i}|\leq\sum_{i\neq j}|a_{i,j}|\right\},~1\leq i\leq n
  \end{equation*}
  and in the union of the discs
  \begin{equation*}
  C'_i=\left\{z\in \mathbb{C};\,|z-a_{i,i}|\leq\sum_{i\neq j}|a_{j,i}|\right\},~1\leq i\leq n.
  \end{equation*}
\end{lemma}

\begin{lemma}\label{lemweporps2}
  The weights of the tempered fractional Laplacian satisfy 
  \begin{equation*}
  \left\{
  \begin{split}
    &\sum_{i=-N}^{i=N}\sum_{j=-N}^{j=N}w^{\beta,\gamma}_{p,q,i,j}>CW^\infty_{p,q}>0;\\
    &w^{\beta,\gamma}_{p,q,i,j}<0,~~~~~(i,j)~\neq (p,q),
  \end{split}
  \right.
  \end{equation*}
  for any given $p,q$.
\end{lemma}
\begin{remark}
The proof of Lemma \ref{lemweporps2} is similar to the proof in \cite{Sun2018}.
\end{remark}

According to Lemmas \eqref{lemmaGersgorin} and \eqref{lemweporps2}, we get that matrix $\mathbf{A_s}$ is strictly diagonally dominant and symmetric positive definite.

\end{document}